\documentclass[a4paper,12pt]{amsart}
\usepackage{amssymb, amsmath, graphicx}
\usepackage[curve]{xypic}
\usepackage{enumerate}
\usepackage{tikz}
\usepackage{color}
\usetikzlibrary{fadings}
\usetikzlibrary{intersections}
\usetikzlibrary{arrows}
\usetikzlibrary{decorations}
\usetikzlibrary{decorations.pathmorphing}
\usetikzlibrary{decorations.pathreplacing}
\usetikzlibrary{decorations.shapes}
\usetikzlibrary{decorations.markings}
\usetikzlibrary{decorations.text}
\usepackage{xr}

\DeclareRobustCommand{\gobblefour}[4]{}

\providecommand{\abs}[1]{\lvert#1\rvert}

\providecommand{\chr}{\textnormal{char}}

\providecommand{\Z}{\mathbb{Z}}
\providecommand{\Q}{\mathbb{Q}}

\providecommand{\LL}{\mathcal{L}}

\providecommand{\K}{\mathbb{K}}

\makeatletter
\newcommand\tint{\mathop{\mathpalette\tb@int{t}}\!\int}
\newcommand\bint{\mathop{\mathpalette\tb@int{b}}\!\int}
\newcommand\tb@int[2]{%
  \sbox\z@{$\m@th#1\int$}%
  \if#2t%
    \rlap{\hbox to\wd\z@{%
      \hfil
      \vrule width .35em height \dimexpr\ht\z@+1.4pt\relax depth -\dimexpr\ht\z@+1pt\relax
      \kern.05em 
    }}
  \else
    \rlap{\hbox to\wd\z@{%
      \vrule width .35em height -\dimexpr\dp\z@+1pt\relax depth \dimexpr\dp\z@+1.4pt\relax
      \hfil
    }}
  \fi
}
\makeatother

\allowdisplaybreaks

\begin{document}

\title{Graded Identities of the Virasoro Algebra}
\author{Fabio Ferrari Ruffino}
\address{Departamento de Matem\'atica - Universidade Federal de S\~ao Carlos - Rod.\ Washington Lu\'is, Km 235 - C.P.\ 676 - 13565-905 S\~ao Carlos, SP, Brasil}
\email{ferrariruffino@ufscar.br}

\begin{abstract}
We study the graded polynomial identities satisfied by the Virasoro algebra over an infinite field, completing the analysis realised by C.\ Fidelis, D.\ Diniz, and P.\ Koshlukov [\emph{$\Z$-graded identities of the Virasoro Algebra}, Journal of Algebra 640 (2024), 401--431].
\end{abstract}

\maketitle

\newtheorem{Theorem}{Theorem}[section]
\newtheorem{Lemma}[Theorem]{Lemma}
\newtheorem{Prop}[Theorem]{Proposition}
\newtheorem{Corollary}[Theorem]{Corollary}
\newtheorem{ThmDef}[Theorem]{Theorem - Definition}

\theoremstyle{definition}
\newtheorem{Rmk}[Theorem]{Remark}
\newtheorem{Rmks}[Theorem]{Remarks}
\newtheorem{Def}[Theorem]{Definition}
\newtheorem{Not}[Theorem]{Notation}
\newtheorem*{Not*}{Notation}
\newtheorem{Ex}[Theorem]{Example}


\tableofcontents


\section{Introduction}

The Virasoro algebra is an infinite-dimensional $\Z$-graded Lie algebra, which is particularly relevant for its applications both in mathematics and in theoretical physics. In the paper \cite{FDK}, the authors show a minimal family of generators of the graded polynomial identities of this algebra over an infinite field. Nevertheless, there is a small computational mistake in formula (6) that, in spite of not affecting the structure of the paper and of the identities found there, it changes the condition to be verified by the degrees of the variables in each case. Consequently, some proofs require quite different computations and verifications both in characteristic zero and in positive characteristic. Furthermore, in the specific cases of characteristic 5 and 7, we need to add a family of identities to the generators, since certain commutators vanish only in this setting, affecting the structure of the arguments in which they are involved.

Here we prove that the graded identities found in \cite{FDK} generate the T-ideal induced by the Virasoro algebra over an infinite field, starting from the correct version of formula (6) and filling the gap in characteristic 5 and 7. Since the present paper is strictly related to \cite{FDK}, we refer to the introduction of the latter for a summary about graded polynomial identities and about the relevance of the Virasoro algebra within the general theory of p.i.-algebras. Moreover, the reader can find a comprehensive survey on graded Lie algebras in the book \cite{EK}.


\section{Identities in Zero Characteristic}

The Virasoro algebra is defined as follows. We fix a field $\K$ and we assume that $\chr \, \K = 0$ in this section. We consider the $\Z$-graded $\K$-vector space $\LL := \bigoplus_{n \in \Z} \LL_{n}$, where $\LL_{n} = \langle L_{n} \rangle$ for every $n \neq 0$ and $\LL_{0} = \langle L_{0}, \hat{c} \rangle$. We endow $\LL$ with the following Lie algebra structure:
\begin{equation}\label{StructVirasoro}
	[L_{m}, L_{n}] := (m-n)L_{m+n} + \delta_{m+n, 0} C_{m} \hat{c} \qquad\quad [L_{m}, \hat{c}] := 0
\end{equation}
where $C_{m} := K m(m^{2}-1)$ for any fixed $K \in \Q^{*}$. We call $T_{\Z}(\LL)$ the T-ideal formed by the graded polynomial identities of $\LL$. Our aim consists in finding a suitable set of generators of $T_{\Z}(\LL)$.

\subsection{First Family}

We have the following identities:
\begin{equation}\label{PolId1}
	\textcolor{blue}{[x_{1}^{a}, x_{2}^{a}]}
\end{equation}
for every $a \in \Z$. That's because each subspace $\LL_{n}$ is one-dimensional, except for $\LL_{0}$, in which the extra-generator $\hat{c}$ is central.

\begin{Prop}\label{PropMultiCom} If $[x_{1}^{a_{1}}, \ldots, x_{n}^{a_{n}}]$ is an identity of $\LL$, then it belongs to the $T$-ideal generated by the identities \eqref{PolId1}.
\end{Prop}
\begin{proof} Induction in $n$. We call $H$ the $T$-ideal generated by the identities \eqref{PolId1}. The thesis is empty for $n = 1$ and obvious for $n = 2$, since \eqref{StructVirasoro} implies that $[L_{a}, L_{b}] = 0$ if and only if $a = b$. Let us suppose that the result holds for $n-1$ and that $[x_{1}^{a_{1}}, \ldots, x_{n}^{a_{n}}]$ is an identity. We set $b := a_{1} + \cdots + a_{n-1}$ and $p^{b} := [x_{1}^{a_{1}}, \ldots, x_{n-1}^{a_{n-1}}]$, so that $[x_{1}^{a_{1}}, \ldots, x_{n}^{a_{n}}] = [p^{b}, x_{n}^{a_{n}}]$. If $p^{b}$ is an identity, then $p^{b} \in H$ by the inductive hypothesis, hence $[p^{b}, x_{n}^{a_{n}}] \in H$ too. Otherwise, any admissible substitution in $p^{b}$ gives a multiple of $L_{b}$ and, if $b = 0$, of $\hat{c}$. The term $\hat{c}$ is immaterial, since it is central. Hence, $[L_{b}, L_{a_{n}}] = 0$, thus $b = a_{n}$. It follows that $[x_{n+1}^{b}, x_{n}^{b}] \in H$, therefore $[p^{b}, x_{n}^{b}] \in H$ too.
\end{proof}

\subsection{Second Family}

We compute the triple commutators of the non-central generators of $\LL$. We have:
\begin{align*}
	[L_{a}, L_{b}, L_{c}] &= \bigl[(a-b)L_{a+b} + \delta_{a+b, 0} C_{a} \hat{c}, L_{c}\bigr] \nonumber \\
	&= (a-b)[L_{a+b}, L_{c}] = (a-b) \bigl((a+b-c)L_{a+b+c} + \delta_{a+b+c, 0} C_{a+b} \hat{c}\bigr) \nonumber \\
	&= (b-a)(c-a-b)L_{a+b+c} + (b-a) \delta_{a+b+c, 0} C_{c}\hat{c}
\end{align*}
In the last equality, we replaced $C_{a+b}$ with $C_{-c}$ because of the term $\delta_{a+b+c, 0}$, and we observed that $C_{-c} = -C_{c}$.\footnote{In formula (6) of \cite{FDK}, the term $(b-a)$ before $\delta_{a+b+c, 0} C_{c}\hat{c}$ is missing. This leads to an inconsistency, since, for example, by evaluating $[L_{1}, L_{1}, L_{-2}]$ through that formula, we do not get $0$.} By exchanging the indices, we get
\begin{align*}
	& [L_{a}, L_{c}, L_{b}] = (c-a)(b-a-c)L_{a+b+c} + (c-a) \delta_{a+b+c, 0} C_{b}\hat{c} \\
	& [L_{b}, L_{c}, L_{a}] = (c-b)(a-b-c)L_{a+b+c} + (c-b) \delta_{a+b+c, 0} C_{a}\hat{c}
\end{align*}
We set $\beta := (b-a)(c-a-b)$ and $\alpha := (c-a)(b-a-c)$. It follows by direct computation that $\alpha - \beta = (c-b)(a-b-c)$, coherently with the Jacobi identity. Hence, we get:
\begin{align}
	& [L_{a}, L_{b}, L_{c}] = \beta L_{a+b+c} + (b-a) \delta_{a+b+c, 0} C_{c}\hat{c} \label{EqLABC} \\
	& [L_{a}, L_{c}, L_{b}] = \alpha L_{a+b+c} + (c-a) \delta_{a+b+c, 0} C_{b}\hat{c} \label{EqLACB} \\
	& [L_{b}, L_{c}, L_{a}] = (\alpha-\beta)L_{a+b+c} + (c-b) \delta_{a+b+c, 0} C_{a}\hat{c} \label{EqLBCA}
\end{align}

\begin{Lemma}\label{LemmaAlphaBetaZero} We have that:
\begin{itemize}
	\item[(i)] $\beta = 0$ if and only if $[L_{a}, L_{b}, L_{c}] = 0$;
	\item[(ii)] $\alpha = 0$ if and only if $[L_{a}, L_{c}, L_{b}] = 0$;
	\item[(iii)] $\alpha = \beta$ if and only if $[L_{b}, L_{c}, L_{a}] = 0$.
\end{itemize}
\end{Lemma}
\begin{proof} The directions $(\Leftarrow)$ immediately follow from equations \eqref{EqLABC}--\eqref{EqLBCA}. If $\beta = 0$, then $a = b$ or $c = a+b$. In both cases, $[L_{a}, L_{b}, L_{c}] = 0$ because of \eqref{PolId1}. The proof is analogous in the other two cases.
\end{proof}

From equations \eqref{EqLABC} and \eqref{EqLACB}, we get:
\begin{equation}\label{IdAlphaBeta}
	\alpha [L_{a}, L_{b}, L_{c}] - \beta [L_{a}, L_{c}, L_{b}] = \delta_{a+b+c, 0} \bigl( \alpha(b-a)C_{c} - \beta(c-a)C_{b}\bigr) \hat{c}
\end{equation}
If $a + b + c \neq 0$, then the term $\delta_{a+b+c, 0}$ leads to the following identity:
\begin{equation}\label{PolId2}
	\textcolor{blue}{\alpha [x_{1}^{a}, x_{2}^{b}, x_{3}^{c}] - \beta [x_{1}^{a}, x_{3}^{c}, x_{2}^{b}]}
\end{equation}

\begin{Def}\label{DefPiTr} A triple $(a, b, c) \in \Z^{3}$ is called
\begin{itemize}
	\item \emph{strong} if $a + b + c \neq 0$; otherwise, it is called \emph{weak};
	\item \emph{regular} if the three entries are pairwise distinct and none of the three is the sum of the other two; otherwise, it is called \emph{special}.
\end{itemize}
\end{Def}

We are going to show that the identity \eqref{PolId2} is non-trivial only for strong regular triples, since, otherwise, it is not an identity or it is a multiple of one of the following three commutators:
\begin{equation}\label{ThreeComm}
	[x_{1}^{a}, x_{2}^{b}, x_{3}^{c}] \qquad\qquad [x_{1}^{a}, x_{3}^{c}, x_{2}^{b}] \qquad\qquad [x_{2}^{b}, x_{3}^{c}, x_{1}^{a}]
\end{equation}
In the latter case, it falls in proposition \ref{PropMultiCom}.

\begin{Prop}\label{PropTriple} A triple $(a, b, c) \in \Z^{3}$ is:
\begin{enumerate}
	\item strong and regular if and only if formula \eqref{PolId2} is a non-vanishing identity and it is not a multiple of a commutator in \eqref{ThreeComm}; in this case, none of such commutators is an identity;
	\item weak and regular if and only if formula \eqref{PolId2} is not an identity; in this case, none of the commutators in \eqref{ThreeComm} is an identity either; and
	\item special if and only if formula \eqref{PolId2} is an identity and it is a multiple of a commutator in \eqref{ThreeComm}; in this case, at least one of the commutators in \eqref{ThreeComm} is an identity (even if formula \eqref{PolId2} vanishes).
\end{enumerate}
Moreover:
\begin{itemize}
	\item[\emph{(i)}] Case (2) holds if and only if $a + b + c = 0$ and the entries of $(\abs{a}, \abs{b}, \abs{c})$ are pairwise distinct (thus, non-vanishing too).
	\item[\emph{(ii)}] In case (3), $(a, b, c)$ is weak if and only if there exists $k \in \Z$ such that $(a, b, c)$ coincides with one of the following triples: $(k, k, -2k)$, $(k, -2k, k)$, $(-2k, k, k)$, $(k, -k, 0)$, $(k, 0, -k)$, $(0, k, -k)$.
	\item[\emph{(iii)}] Formula \eqref{PolId2} vanishes if and only if there exists $k \in \Z$ such that $(a, b, c)$ takes one of the following forms: $(k, k, k)$, $(k, k, 0)$, $(k, 0, k)$, $(0, k, k)$.
\end{itemize}
\end{Prop}
\begin{proof} \emph{Step 1:} It immediately follows from the explicit expressions of $\alpha$, $\beta$, and $\alpha-\beta$, shown before equation \eqref{EqLABC}, that $(a, b, c)$ is regular if and only if $\alpha \neq 0$, $\beta \neq 0$ and $\alpha \neq \beta$. By lemma \ref{LemmaAlphaBetaZero}, this is equivalent to state that none of the commutators in \eqref{ThreeComm} is an identity.

\vspace{3pt} \emph{Step 2:} It is enough to prove the direction $(\Rightarrow)$ in each item (1)--(3), since the opposite implication follows from the fact that a triple $(a, b, c)$ necessarily falls in only one of the three cases considered.

\vspace{3pt} (1) Since $(a, b, c)$ is strong, \eqref{PolId2} is an identity by formula \eqref{IdAlphaBeta}. Moreover, since $(a, b, c)$ is regular, step 1 completes the proof.

\vspace{3pt} (2) Since $(a, b, c)$ is weak, formula \eqref{IdAlphaBeta} is an identity if and only if $\alpha(b-a)C_{c} - \beta(c-a)C_{b} = 0$. We get the following equivalent conditions:
\begin{align*}
	& (c-a)(b-a-c)(b-a)c(c^{2}-1) - (b-a)(c-a-b)(c-a)b(b^{2}-1) = 0 \\
	& (c-a)2b(b-a)c(c^{2}-1) - (b-a)2c(c-a)b(b^{2}-1) = 0 \\
	& (c-a)b(b-a)c(c^{2}-1-b^{2}+1) = 0 \\
	& (c-a)b(b-a)c(c-b)(c+b) = 0 \\
	& (c-a)b(b-a)c(c-b)a = 0
\end{align*}
This means that $(a, b, c)$ has to be special, against the hypothesis; in fact, for example, the condition $a=0$ is equivalent to $b+c=a$ under the assumption $a+b+c=0$. Step 1 completes the proof.

\vspace{3pt} (3) Let $(a, c, b)$ be special. If it is strong, then \eqref{PolId2} is an identity by formula \eqref{IdAlphaBeta}. Otherwise, the computation in item 2 shows that \eqref{PolId2} is an identity anyway. By step 1, $\beta = 0$, $\alpha = 0$, or $\alpha = \beta$; in the first case, \eqref{PolId2} is a multiple of $[x_{1}^{a}, x_{2}^{b}, x_{3}^{c}]$; in the second case, it is a multiple of $[x_{1}^{a}, x_{3}^{c}, x_{2}^{b}]$; in the third case, because of the Jacobi identity, it is a multiple of $[x_{2}^{b}, x_{3}^{c}, x_{1}^{a}]$. Again by step 1, one of these commutators is an identity.

\vspace{3pt} \emph{Step 3:} The reader can easily verify statements (i)--(iii) by direct computation.
\end{proof}

Because of the previous proposition, \eqref{PolId2} is a non-trivial identity if and only if $(a, b, c)$ is a strong regular triple, as we anticipated above.

\begin{Not} We denote by $K$ the $T$-ideal generated by the identities \eqref{PolId1} and \eqref{PolId2}, assuming $(a, b, c)$ strong and regular in the latter case. Moreover, given a $T$-ideal $I$ and two polynomials $p, q \notin I$, we denote by $p \sim_{I} q$ the fact that there exists $\lambda \in \K$ such that $p - \lambda q \in I$ (necessarily, $\lambda \neq 0$).
\end{Not}

\begin{Lemma}\label{LemmaTripleNotId} If $p := [x_{1}^{a}, x_{2}^{b}, x_{3}^{c}]$ and $q := [x_{1}^{a}, x_{3}^{c}, x_{2}^{b}]$ are not identities and the triple $(a, b, c)$ is strong or special, then $p \sim_{K} q$.
\end{Lemma}
\begin{proof} If $(a, b, c)$ is is special, then, considering that $p$ and $q$ are not identities by hypothesis, proposition \ref{PropTriple}-3 and the Jacobi identity imply that $[x_{2}^{b}, x_{3}^{c}, x_{1}^{a}] = q - p$ is an identity, hence $p \sim_{K} q$ by proposition \ref{PropMultiCom}. Otherwise, $(a, b, c)$ is strong and regular, hence identity \eqref{PolId2} implies that $p \sim_{K} q$.
\end{proof}

\begin{Prop}\label{PropMultiL} Every multilinear identity of degree $3$ belongs to $K$.
\end{Prop}
\begin{proof} We can express any such identity in the form
\begin{equation}\label{IdDeg3}\tag{$\star$}
	\lambda[x_{1}^{a}, x_{2}^{b}, x_{3}^{c}] - \mu[x_{1}^{a}, x_{3}^{c}, x_{2}^{b}],
\end{equation}
excluding the trivial case $\lambda = \mu = 0$. If $\lambda = 0$ or $\mu = 0$, then \eqref{IdDeg3} is a triple commutator. Moreover, if $\lambda = \mu$, then the Jacobi identity implies that \eqref{IdDeg3} is a triple commutator as well. Thus, in these cases, the thesis follows from proposition \ref{PropMultiCom}. For this reason, we can write \eqref{IdDeg3} in the form
\begin{equation}\label{IdDeg3B}\tag{$\star\star$}
	[x_{1}^{a}, x_{2}^{b}, x_{3}^{c}] + \xi[x_{1}^{a}, x_{3}^{c}, x_{2}^{b}], \quad \xi \neq 0, -1
\end{equation}
without loss of generality. We distinguish two cases.

\vspace{3pt} I. \emph{$(a, b, c)$ is strong or special.} If $p := [x_{1}^{a}, x_{2}^{b}, x_{3}^{c}]$ is an identity, then $q := [x_{1}^{a}, x_{3}^{c}, x_{2}^{b}]$ is an identity too because \eqref{IdDeg3B} is, so that proposition \ref{PropMultiCom} implies the thesis; the same argument holds if $q$ is an identity. Therefore, we assume that $p$ and $q$ are not identities. Then, lemma \ref{LemmaTripleNotId} implies that $q = \nu p + k$, with $\nu \in \K$ and $k \in K$, hence $\eqref{IdDeg3B} = (1 + \xi\nu)p + \xi k$. Since $p$ is not an identity, while \eqref{IdDeg3B} and $k$ are, we have $1 + \xi\nu = 0$, thus $\eqref{IdDeg3B} = \xi k \in K$.

\vspace{3pt} II. \emph{$(a, b, c)$ is weak and regular.} From formulas \eqref{EqLABC} and \eqref{EqLACB} we get
	\[0 = [L_{a}, L_{b}, L_{c}] + \xi[L_{a}, L_{c}, L_{b}] = (\beta + \xi\alpha)L_{0} + \bigl((b-a)C_{c} + \xi(c-a)C_{b}\bigr)\hat{c}.
\]
Since $\alpha \neq 0$ by lemma \ref{LemmaAlphaBetaZero}, the coefficient of $L_{0}$ implies $\xi = -\frac{\beta}{\alpha}$, hence the coefficient of $\hat{c}$ implies $\alpha(b-a)C_{c} - \beta(c-a)C_{b} = 0$. By formula \eqref{IdAlphaBeta}, this is exactly the condition that makes \eqref{PolId2} an identity, that is forbidden by proposition \ref{PropTriple}-2. Hence, \eqref{IdDeg3B} cannot be an identity.
\end{proof}

\subsection{Third Family}

Formula \eqref{IdAlphaBeta} shows that \eqref{PolId2} is central, hence we get the following identity for every $d \in \Z$:
\begin{equation}\label{PolId3}
	\textcolor{blue}{\alpha [x_{1}^{a}, x_{2}^{b}, x_{3}^{c}, x_{4}^{d}] - \beta [x_{1}^{a}, x_{3}^{c}, x_{2}^{b}, x_{4}^{d}]}.
\end{equation}
This identity is non-trivial when \eqref{PolId2} is \emph{not} an identity---that is, when $(a, b, c)$ is weak and regular. Moreover, if $d = a + b + c$, then the two quadruple commutators are both identities following from \eqref{PolId1}, hence \eqref{PolId3} is trivial.

\begin{Not} We denote by $J$ the $T$-ideal generated by the identities \eqref{PolId1}, \eqref{PolId2} with $(a, b, c)$ strong and regular, and \eqref{PolId3} with $(a, b, c)$ weak and regular and $d \neq a+b+c$.
\end{Not}

\begin{Lemma}\label{LemmaQuadrNotId} Let $p := [x_{1}^{a}, x_{2}^{b}, x_{3}^{c}, x_{4}^{d}]$ be a quadruple commutator that is not an identity. We set $q := [x_{1}^{a}, x_{3}^{c}, x_{2}^{b}, x_{4}^{d}]$ and $r := [x_{1}^{a}, x_{2}^{b}, x_{4}^{d}, x_{3}^{c}]$.
\begin{itemize}
	\item[(i)] If $q$ is not an identity (equivalently, if $[x_{1}^{a}, x_{3}^{c}, x_{2}^{b}]$ is not an identity), then $p \sim_{J} q$.
	\item[(ii)] If $a + b + c + d \neq 0$ and $r$ is not an identity, then $p \sim_{J} r$.
\end{itemize}
\end{Lemma}
\begin{proof} (i) If $(a, b, c)$ is strong or special, then lemma \ref{LemmaTripleNotId} applies to $[x_{1}^{a}, x_{2}^{b}, x_{3}^{c}]$. If $(a, b, c)$ is weak and regular, then identity \eqref{PolId3} applies (necessarily $d \neq a+b+c$, since $p$ is not an identity by hypothesis).

\vspace{3pt} (ii) The triple $(a + b, c, d)$ cannot be weak because of the hypothesis $a + b + c + d \neq 0$, hence lemma \ref{LemmaTripleNotId} applies to $[[x_{1}^{a}, x_{2}^{b}], x_{3}^{c}, x_{4}^{d}]$.
\end{proof}

In any Lie algebra, the obvious identity $[[x_{1}^{a}, x_{2}^{b}], [x_{3}^{c}, x_{4}^{d}]] + [[x_{3}^{c}, x_{4}^{d}], [x_{1}^{a}, x_{2}^{b}]] = 0$ is equivalent to
\begin{equation}\label{IdLie4}
	[x_{1}^{a}, x_{2}^{b}, x_{3}^{c}, x_{4}^{d}] + [x_{2}^{b}, x_{1}^{a}, x_{4}^{d}, x_{3}^{c}] + [x_{4}^{d}, x_{3}^{c}, x_{2}^{b}, x_{1}^{a}] + [x_{3}^{c}, x_{4}^{d}, x_{1}^{a}, x_{2}^{b}] = 0.
\end{equation}
This can be easily proven by applying the Jacobi identity to the triples $x_{1}^{a}, x_{2}^{b}, [x_{3}^{c}, x_{4}^{d}]$ and $x_{3}^{c}, x_{4}^{d}, [x_{1}^{a}, x_{2}^{b}]$.

\begin{Prop}\label{PropMultId4Sp} Any multilinear identity $p(x_{1}^{a}, x_{2}^{b}, x_{3}^{c}, x_{4}^{d})$ such that $a + b + c + d \neq 0$ belongs to $J$.
\end{Prop}
\begin{proof} We can suppose that one of the indices $a, b, c, d$ is not the sum of the others, otherwise we would get a linear system of four equations whose only solution is $(0, 0, 0, 0)$. Up to a permutation, we assume that
\begin{equation}\label{CondABCD1}\tag{$*$}
	d \neq a + b + c.
\end{equation}
We can express $p$ as a linear combination of monomials of the form
	\[N_{\sigma} := [x_{4}^{d}, x_{\sigma(1)}^{\sigma(a)}, x_{\sigma(2)}^{\sigma(b)}, x_{\sigma(3)}^{\sigma(c)}],
\]
with $\sigma \in S_{3}$. If every such permutation is an identity, then proposition \ref{PropMultiCom} implies the result. Hence, up to a permutation of $\{a, b, c\}$, we assume that $[x_{4}^{d}, x_{1}^{a}, x_{2}^{b}, x_{3}^{c}]$ is not an identity. Thus:
\begin{equation}\label{CondABCD2}\tag{$**$}
	\textnormal{(i)} \;\, d \neq a \qquad \textnormal{(ii)} \;\, d + a \neq b \qquad \textnormal{(iii)} \;\, d + a + b \neq c
\end{equation}
Our aim consists in proving that, for every $\sigma \in S_{3}$, the monomial $N_{\sigma}$ is an identity (hence, proposition \ref{PropMultiCom} applies) or $N_{\sigma} \sim_{J} [x_{4}^{d}, x_{1}^{a}, x_{2}^{b}, x_{3}^{c}]$. It follows that the identity $p$ satisfies $p = j + \lambda [x_{4}^{d}, x_{1}^{a}, x_{2}^{b}, x_{3}^{c}]$, where $j \in J$ and $\lambda \in \K$. Since $j$ is an identity too, while $[x_{4}^{d}, x_{1}^{a}, x_{2}^{b}, x_{3}^{c}]$ is not, we get $\lambda = 0$, thus $p \in J$, as desired. We have to consider the five non-trivial elements of $S_{3}$, assuming in each case that $N_{\sigma}$ is not an identity.

\vspace{3pt} I. $N_{\sigma} = [x_{4}^{d}, x_{2}^{b}, x_{1}^{a}, x_{3}^{c}]$. Since $[x_{4}^{d}, x_{1}^{a}, x_{2}^{b}, x_{3}^{c}]$ is not an identity by hypothesis, the result immediately follows from lemma \ref{LemmaQuadrNotId}-i.

\vspace{3pt} II. $N_{\sigma} = [x_{4}^{d}, x_{1}^{a}, x_{3}^{c}, x_{2}^{b}]$. Since $[x_{4}^{d}, x_{1}^{a}, x_{2}^{b}, x_{3}^{c}]$ is not an identity by hypothesis, the result immediately follows from lemma \ref{LemmaQuadrNotId}-ii.

\vspace{3pt} III. $N_{\sigma} = [x_{4}^{d}, x_{2}^{b}, x_{3}^{c}, x_{1}^{a}]$. We deal with this case in three steps.

\vspace{3pt} \emph{Step 1:} If $[x_{4}^{d}, x_{2}^{b}, x_{1}^{a}, x_{3}^{c}]$ is not an identity, then by lemma \ref{LemmaQuadrNotId}-ii we fall in case I. Otherwise, we have $d = b$, $d+b = a$, or $d+b+a=c$; the latter is forbidden by \eqref{CondABCD2}-iii and $d = b$ makes $N_{\sigma}$ an identity. Thus, only $d+b=a$ is possible.

\vspace{3pt} \emph{Step 2:} By equation \eqref{IdLie4} and assuming $d+b=a$, we get:
\begin{align}
	& [x_{4}^{d}, x_{2}^{b}, x_{3}^{c}, x_{1}^{a}] + [x_{2}^{b}, x_{4}^{d}, x_{1}^{a}, x_{3}^{c}] + [x_{1}^{a}, x_{3}^{c}, x_{2}^{b}, x_{4}^{d}] + [x_{3}^{c}, x_{1}^{a}, x_{4}^{d}, x_{2}^{b}] = 0 \nonumber \\
	& N_{\sigma} + 0 + [x_{1}^{a}, x_{3}^{c}, x_{2}^{b}, x_{4}^{d}] - [x_{1}^{a}, x_{3}^{c}, x_{4}^{d}, x_{2}^{b}] = 0 \nonumber \\
	& N_{\sigma} = [x_{1}^{a}, x_{3}^{c}, x_{4}^{d}, x_{2}^{b}] - [x_{1}^{a}, x_{3}^{c}, x_{2}^{b}, x_{4}^{d}] \label{EqStep3IV}
\end{align}
If $[x_{1}^{a}, x_{3}^{c}, x_{4}^{d}, x_{2}^{b}]$ is not an identity, then by lemma \ref{LemmaQuadrNotId}-ii or proposition \ref{PropMultiCom} we have $N_{\sigma} \sim_{J} [x_{1}^{a}, x_{3}^{c}, x_{4}^{d}, x_{2}^{b}]$. Otherwise, we have $a=c$, $a+c=d$ or $a+c+d=b$; the first case makes $N_{\sigma}$ an identity by \eqref{EqStep3IV}. Unless $a = d$ or $a+d=c$, lemma \ref{LemmaQuadrNotId}-i implies that $N_{\sigma} \sim_{J} [x_{1}^{a}, x_{4}^{d}, x_{3}^{c}, x_{2}^{b}] = -[x_{4}^{d}, x_{1}^{a}, x_{3}^{c}, x_{2}^{b}]$. The condition $a = d$ is forbidden by \eqref{CondABCD2}-i. When no obstructions are present, we fall in case II.

\vspace{3pt} \emph{Step 3:} Let us see when the two previous steps do not apply. Failure of step 1 is equivalent to $d+b=a$; failure of step 2 is equivalent to $a+c=d$, $a+c+d=b$, or $a+d=c$. We analyse these cases.
\begin{enumerate}
	\item\label{FirstCase} $d+b=a$ and $a+c=d$: we have $(a, b, c, d) = (a, a-d, d-a, d)$. By equation \eqref{EqStep3IV}, we have $N_{\sigma} = -[x_{1}^{a}, x_{3}^{d-a}, x_{2}^{a-d}, x_{4}^{d}]$. Lemma \ref{LemmaQuadrNotId}-i implies $N_{\sigma} \sim_{J} [x_{1}^{a}, x_{2}^{a-d}, x_{3}^{d-a}, x_{4}^{d}]$, unless $a = a-d$ or $a+(a-d)=(d-a)$. The first case leads to $d = 0$, against \eqref{CondABCD2}-ii; the second case is equivalent to $3a=2d$, leading to $(a, b, c, d) = (2c, -c, c, 3c)$. We will analyse this case below. Now $N_{\sigma} \sim_{J} [x_{1}^{a}, x_{2}^{a-d}, x_{4}^{d}, x_{3}^{d-a}]$ by lemma \ref{LemmaQuadrNotId}-ii, since $a=(a-d)$ has already been excluded, $a+(a-d)=d$ leads to $a = d$, against \eqref{CondABCD2}-i, and $a+b+d=c$ is excluded by \eqref{CondABCD2}-iii. Lastly, $N_{\sigma} \sim_{J} [x_{1}^{a}, x_{4}^{d}, x_{2}^{a-d}, x_{3}^{d-a}] = -[x_{4}^{d}, x_{1}^{a}, x_{2}^{a-d}, x_{3}^{d-a}]$, as desired, by lemma \ref{LemmaQuadrNotId}-i, since $a=d$ is excluded by \eqref{CondABCD2}-i and $a+d=a-d$ leads to $d = 0$, that has already been excluded.
	\vspace{3pt} \item $(a, b, c, d) = (2c, -c, c, 3c)$: by applying iteratively lemma \ref{LemmaQuadrNotId}, we get $N_{\sigma} = [x_{4}^{3c}, x_{2}^{-c}, x_{3}^{c}, x_{1}^{2c}] \sim_{J} [x_{4}^{3c}, x_{3}^{c}, x_{2}^{-c}, x_{1}^{2c}] \sim_{J} [x_{4}^{3c}, x_{3}^{c}, x_{1}^{2c}, x_{2}^{-c}] \sim_{J} [x_{4}^{3c}, x_{1}^{2c}, x_{3}^{c}, x_{2}^{-c}] \sim_{J} [x_{4}^{3c}, x_{1}^{2c}, x_{2}^{-c}, x_{3}^{c}]$, as required.
	\vspace{3pt} \item $d+b=a$ and $a+c+d=b$: we have $(a, b, c, d) = (a, a-d, -2d, d)$. By equation \eqref{EqStep3IV}, $N_{\sigma} = -[x_{1}^{a}, x_{3}^{-2d}, x_{2}^{a-d}, x_{4}^{d}]$. By lemma \ref{LemmaQuadrNotId}-i, $N_{\sigma} \sim_{J} [x_{1}^{a}, x_{2}^{a-d}, x_{3}^{-2d}, x_{4}^{d}]$, unless $a=a-d$ or $a+(a-d)=-2d$; in the first case, $d = 0$, against \eqref{CondABCD2}-ii; in the second case, that we will analyse below, $d=-2a$. By lemma \ref{LemmaQuadrNotId}-ii, $N_{\sigma} \sim_{J} [x_{1}^{a}, x_{2}^{a-d}, x_{4}^{d}, x_{3}^{-2d}]$, unless $a+(a-d)=d$---that is, $a = d$---against \eqref{CondABCD2}-i, or $a+(a-d)+d=-2d$---that is, $a = -d$---against \eqref{CondABCD2}-iii. Lastly, by lemma \ref{LemmaQuadrNotId}-i, $N_{\sigma} \sim_{J} [x_{1}^{a}, x_{4}^{d}, x_{2}^{a-d}, x_{3}^{-2d}] = -[x_{4}^{d}, x_{1}^{a}, x_{2}^{a-d}, x_{3}^{-2d}]$, as desired, unless $d = a$, against \eqref{CondABCD2}-i, or $d+a=a-d$---that is, $d = 0$---against \eqref{CondABCD2}-ii.
	\vspace{3pt} \item $d+b=a$, $a+c+d=b$, and $d=-2a$: we have $(a, b, c, d) = (a, 3a, 4a, -2a)$, hence $N_{\sigma} = [x_{4}^{-2a}, x_{2}^{3a}, x_{3}^{4a}, x_{1}^{a}]$. By lemma \ref{LemmaQuadrNotId}-i, $N_{\sigma} \sim_{J} [x_{4}^{-2a}, x_{3}^{4a}, x_{2}^{3a}, x_{1}^{a}]$. By applying the Jacobi identity, we get $N_{\sigma} \sim_{J} [x_{2}^{3a}, x_{1}^{a}, [x_{4}^{-2a}, x_{3}^{4a}]] = [x_{4}^{-2a}, x_{3}^{4a}, [x_{1}^{a}, x_{2}^{3a}]]$. By lemma \ref{LemmaTripleNotId}, we have $N_{\sigma} \sim_{J} [x_{4}^{-2a}, [x_{1}^{a}, x_{2}^{3a}], x_{3}^{4a}] = -[x_{1}^{a}, x_{2}^{3a}, x_{4}^{-2a}, x_{3}^{4a}]$. By lemma \ref{LemmaQuadrNotId}-i, $N_{\sigma} \sim_{J} [x_{1}^{a}, x_{4}^{-2a}, x_{2}^{3a}, x_{3}^{4a}] = -[x_{4}^{-2a}, x_{1}^{a}, x_{2}^{3a}, x_{3}^{4a}]$, as required.
	\vspace{3pt} \item $d+b=a$ and $a+d=c$: we get $(a, b, c, d) = (a, a-d, a+d, d)$, hence $N_{\sigma} = [x_{4}^{d}, x_{2}^{a-d}, x_{3}^{a+d}, x_{1}^{a}] = -[x_{2}^{a-d}, x_{4}^{d}, x_{3}^{a+d}, x_{1}^{a}]$. We have $N_{\sigma} \sim_{J} [x_{2}^{a-d}, x_{3}^{a+d}, x_{4}^{d}, x_{1}^{a}]$ by lemma \ref{LemmaQuadrNotId}-i, unless $a-d=a+d$---that is, $d = 0$---against \eqref{CondABCD2}-ii, or $(a-d)+(a+d)=d$---that is, $d = 2a$---that we will analyse below. Lemma \ref{LemmaQuadrNotId}-ii implies that $N_{\sigma} \sim_{J} [x_{2}^{a-d}, x_{3}^{a+d}, x_{1}^{a}, x_{4}^{d}]$, unless $a-d=a+d$, that we have already excluded, $(a-d)+(a+d)=a$---that is, $a = 0$---falling in case \eqref{FirstCase}, or $(a-d)+(a+d)+a=d$---that is, $d = 3a$---against \eqref{CondABCD1}. By lemma \ref{LemmaQuadrNotId}-i, we have $N_{\sigma} \sim_{J} [x_{2}^{a-d}, x_{1}^{a}, x_{3}^{a+d}, x_{4}^{d}] = -[x_{1}^{a}, x_{2}^{a-d}, x_{3}^{a+d}, x_{4}^{d}]$, unless $a-d=a$, that we have already excluded, or $(a-d)+a=a+d$---that is, $a=2d$---that makes $N_{\sigma}$ an identity (since $d=b$). By lemma \ref{LemmaQuadrNotId}-ii, we have $N_{\sigma} \sim_{J} [x_{1}^{a}, x_{2}^{a-d}, x_{4}^{d}, x_{3}^{a+d}]$, since the conditions $a+(a-d)=d$ and $a+(a-d)+d=(a+d)$ are both equivalent to $a=d$, that is forbidden by \eqref{CondABCD2}-i. Lastly, lemma \ref{LemmaQuadrNotId}-i implies that $N_{\sigma} \sim_{J} [x_{1}^{a}, x_{4}^{d}, x_{2}^{a-d}, x_{3}^{a+d}] = -[x_{4}^{d}, x_{1}^{a}, x_{2}^{a-d}, x_{3}^{a+d}]$, as desired, since \eqref{CondABCD2} rules out any obstruction.
	\vspace{3pt} \item $d+b=a$, $a+d=c$ and $d = 2a$: we get $(a, b, c, d) = (a, -a, 3a, 2a)$, hence $N_{\sigma} = [x_{4}^{2a}, x_{2}^{-a}, x_{3}^{3a}, x_{1}^{a}]$. By applying lemma \ref{LemmaQuadrNotId}, we get $N_{\sigma} \sim_{J} [x_{4}^{2a}, x_{3}^{3a}, x_{2}^{-a}, x_{1}^{a}] \sim_{J} [x_{4}^{2a}, x_{3}^{3a}, x_{1}^{a}, x_{2}^{-a}] = -[x_{3}^{3a}, x_{4}^{2a}, x_{1}^{a}, x_{2}^{-a}] \sim_{J} [x_{3}^{3a}, x_{1}^{a}, x_{4}^{2a}, x_{2}^{-a}] = -[x_{1}^{a}, x_{3}^{3a}, x_{4}^{2a}, x_{2}^{-a}]$ $\sim_{J} [x_{1}^{a}, x_{3}^{3a}, x_{2}^{-a}, x_{4}^{2a}] \sim_{J} [x_{1}^{a}, x_{2}^{-a}, x_{3}^{3a}, x_{4}^{2a}] \sim_{J} [x_{1}^{a}, x_{2}^{-a}, x_{4}^{2a}, x_{3}^{3a}] \sim_{J} [x_{1}^{a}, x_{4}^{2a}, x_{2}^{-a},$ $x_{3}^{3a}] = -[x_{4}^{2a}, x_{1}^{a}, x_{2}^{-a}, x_{3}^{3a}]$, as required.
\end{enumerate}

\vspace{3pt} IV. $N_{\sigma} = [x_{4}^{d}, x_{3}^{c}, x_{1}^{a}, x_{2}^{b}]$. We deal with this case in four steps.

\vspace{3pt} \emph{Step 1:} If $[x_{4}^{d}, x_{1}^{a}, x_{3}^{c}]$ is not an identity, then by lemma \ref{LemmaQuadrNotId}-i we fall in case II. Otherwise, $d = a$ our $d+a = c$. By \eqref{CondABCD2}-i, only $d+a=c$ is possible.

\vspace{3pt} \emph{Step 2:} If we can apply lemma \ref{LemmaQuadrNotId}-ii and i respectively, we get $N_{\sigma} \sim_{J} [x_{4}^{d}, x_{3}^{c}, x_{2}^{b}, x_{1}^{a}] \sim_{J} [x_{4}^{d}, x_{2}^{b}, x_{3}^{c}, x_{1}^{a}]$, falling in case III. About the first equivalence, the possible obstructions are $d = c$, $d+c=b$, or $d+c+b=a$. We can neglect the case $d = c$, since it makes $N_{\sigma}$ an identity. About the second equivalence, the possible obstructions are $d = b$ and $d+b=c$.

\vspace{3pt} \emph{Step 3:} From equation \eqref{IdLie4} we get:
\begin{align}
	& [x_{4}^{d}, x_{3}^{c}, x_{1}^{a}, x_{2}^{b}] + [x_{3}^{c}, x_{4}^{d}, x_{2}^{b}, x_{1}^{a}] + [x_{2}^{b}, x_{1}^{a}, x_{3}^{c}, x_{4}^{d}] + [x_{1}^{a}, x_{2}^{b}, x_{4}^{d}, x_{3}^{c}] = 0 \nonumber \\
	& [x_{4}^{d}, x_{3}^{c}, x_{1}^{a}, x_{2}^{b}] - [x_{4}^{d}, x_{3}^{c}, x_{2}^{b}, x_{1}^{a}] = [x_{1}^{a}, x_{2}^{b}, x_{3}^{c}, x_{4}^{d}] - [x_{1}^{a}, x_{2}^{b}, x_{4}^{d}, x_{3}^{c}] \label{EqStep3}
\end{align}
We call $A$ and $B$ respectively the l.h.s.\ and the r.h.s.\ of the latter equality. The Jacobi identity implies that $A = B = [[x_{4}^{d}, x_{3}^{c}], [x_{1}^{a}, x_{2}^{b}]]$. We make two assumptions: (i) $A$ (or $B$) is not an identity and (ii) the commutator $[x_{1}^{a}, x_{2}^{b}, x_{4}^{d}, x_{3}^{c}]$ is not an identity. By (i), lemma \ref{LemmaQuadrNotId}-ii applied to $[x_{4}^{d}, x_{3}^{c}, x_{2}^{b}, x_{1}^{a}]$ (if the latter is not an identity) implies that $A \sim_{J} N_{\sigma}$, hence $N_{\sigma} \sim_{J} B$. By (ii), lemma \ref{LemmaQuadrNotId}-ii applied to $[x_{1}^{a}, x_{2}^{b}, x_{3}^{c}, x_{4}^{d}]$ (if the latter is not an identity) implies that $B \sim_{J} [x_{1}^{a}, x_{2}^{b}, x_{4}^{d}, x_{3}^{c}]$. Since $a \neq b$ by (i) and $a+d \neq b$ by \eqref{CondABCD2}-ii, lemma \ref{LemmaQuadrNotId}-i implies that $B \sim_{J} [x_{1}^{a}, x_{4}^{d}, x_{2}^{b}, x_{3}^{c}] = -[x_{4}^{d}, x_{1}^{a}, x_{2}^{b}, x_{3}^{c}]$, hence we get the result. Assumption (i) fails if $a = b$, $c = d$, or $a+b=c+d$; again, $c=d$ can be neglected, since $N_{\sigma}$ would be an identity. Assumption (ii) fails if $a=b$, $a+b=d$, or $a+b+d=c$, the latter case being forbidden by \eqref{CondABCD2}-iii.

\vspace{3pt} \emph{Step 4:} Let us see when none of the previous steps applies. Failure of step 1 is equivalent to $d+a=c$; failure of step 2 is equivalent to $d+c=b$, $d+c+b=a$, $d = b$, or $d+b=c$; failure of step 3 is equivalent to $a=b$, $a+b=c+d$, or $a+b=d$. We analyse these cases.
\begin{enumerate}
	\item $d+a=c$ and $a=b$ (any choice for step 2). We have $(a, b, c, d) = (a, a, a+d, d)$. Hence, $N_{\sigma} = [x_{4}^{d}, x_{3}^{a+d}, x_{1}^{a}, x_{2}^{a}]$. By applying the Jacobi identity to $[x_{4}^{d}, x_{3}^{a+d}, x_{1}^{a}]$, we get $N_{\sigma} = -[x_{3}^{a+d}, x_{1}^{a}, x_{4}^{d}, x_{2}^{a}] = [x_{4}^{d}, [x_{3}^{a+d}, x_{1}^{a}], x_{2}^{a}]$. We can apply lemma \ref{LemmaTripleNotId}, because $a+b+c+d \neq 0$ by hypothesis and $[x_{4}^{d}, x_{2}^{a}, [x_{3}^{a+d}, x_{1}^{a}]]$ cannot be an identity, since $d = a$ is forbidden by \eqref{CondABCD2}-i and $d+a=a+d+a$---that is, $a = 0$---makes $N_{\sigma}$ an identity (since $d = a+d$). Therefore, $N_{\sigma} \sim_{J} [x_{4}^{d}, x_{2}^{a}, [x_{3}^{a+d}, x_{1}^{a}]]$.
	
	Moreover, $[x_{4}^{d}, x_{1}^{a}, x_{2}^{b}, x_{3}^{c}] = [x_{4}^{d}, x_{1}^{a}, x_{2}^{a}, x_{3}^{a+d}] \sim_{J} [x_{4}^{d}, x_{2}^{a}, x_{1}^{a}, x_{3}^{a+d}]$. By applying the Jacobi identity, we get $-[x_{1}^{a}, x_{3}^{a+d}, [x_{4}^{d}, x_{2}^{a}]] = [x_{4}^{d}, x_{2}^{a}, [x_{3}^{a+d}, x_{1}^{a}]]$. The previous paragraph implies the result.
	\vspace{3pt}\label{FirstCaseIV} \item $d+a=c$, $d+c=b$, and $a+b=c+d$. We get $(a, b, c, d) = (0, 2d, d, d)$, hence $N_{\sigma}$ is an identity (since $d=c$).
	\vspace{3pt} \item $d+a=c$, $d+c=b$, and $a+b=d$. We get $(a, b, c, d) = (-c, 3c, c, 2c)$, hence $N_{\sigma} = [x_{4}^{2c}, x_{3}^{c}, x_{1}^{-c}, x_{2}^{3c}]$. Equation \eqref{EqStep3} implies that $N_{\sigma} = [x_{1}^{-c}, x_{2}^{3c}, x_{3}^{c}, x_{4}^{2c}]$. Now we apply lemma \ref{LemmaQuadrNotId} as follows: $N_{\sigma} = -[x_{2}^{3c}, x_{1}^{-c}, x_{3}^{c}, x_{4}^{2c}] \sim_{J} [x_{2}^{3c}, x_{1}^{c}, x_{3}^{-c}, x_{4}^{2c}] \sim_{J} [x_{2}^{3c}, x_{1}^{c}, x_{4}^{2c}, x_{3}^{-c}] \sim_{J} [x_{2}^{3c}, x_{4}^{2c}, x_{1}^{c}, x_{3}^{-c}] = -[x_{4}^{2c}, x_{2}^{3c}, x_{1}^{c}, x_{3}^{-c}] \sim_{J} [x_{4}^{2c}, x_{2}^{3c}, x_{1}^{-c}, x_{3}^{c}] \sim_{J} [x_{4}^{2c}, x_{1}^{-c}, x_{2}^{3c}, x_{3}^{c}]$, as required.
	\vspace{3pt} \item $d+a=c$, $d+c+b=a$, and $a+b=c+d$: we get $(a, b, c, d) = (a, 0, a, 0)$, against \eqref{CondABCD2}-iii.
	\vspace{3pt} \item $d+a=c$, $d+c+b=a$, and $a+b=d$: we get $(a, b, c, d) = (3d, -2d, 4d, d)$, hence $N_{\sigma} = [x_{4}^{d}, x_{3}^{4d}, x_{1}^{3d}, x_{2}^{-2d}]$. By applying the Jacobi identity to $[x_{4}^{d}, x_{3}^{4d}, x_{2}^{3d}]$, we get $N_{\sigma} = -[x_{3}^{4d}, x_{2}^{3d}, x_{4}^{d}, x_{1}^{-2d}] = [x_{2}^{3d}, x_{3}^{4d}, x_{4}^{d}, x_{1}^{-2d}]$. From identity \eqref{IdLie4} we get
		\[[x_{1}^{3d}, x_{3}^{4d}, x_{4}^{d}, x_{2}^{-2d}] + [x_{3}^{4d}, x_{1}^{3d}, x_{2}^{-2d}, x_{4}^{d}] = -[x_{2}^{-2d}, x_{4}^{d}, x_{3}^{4d}, x_{1}^{3d}] - [x_{4}^{d}, x_{2}^{-2d}, x_{1}^{3d}, x_{3}^{4d}].
	\]
	The two sides are not identities (since $3d \neq 4d$, $d \neq -2d$, and $3d+4d \neq d-2d$). In the l.h.s., we have $[x_{3}^{4d}, x_{1}^{3d}, x_{2}^{-2d}, x_{4}^{d}] = -[x_{1}^{3d}, x_{3}^{4d}, x_{2}^{-2d}, x_{4}^{d}] \sim_{J} [x_{1}^{3d}, x_{3}^{4d}, x_{4}^{d}, x_{2}^{-2d}]$ by lemma \ref{LemmaQuadrNotId}-ii, thus $\textnormal{l.h.s.} \sim_{J} N_{\sigma}$ by the computation above. In the r.h.s., the first term vanishes (since $-2d+d+4d=3d$), hence we have $N_{\sigma} \sim_{J} [x_{4}^{d}, x_{2}^{-2d}, x_{1}^{3d}, x_{3}^{4d}]$. From lemma \ref{LemmaQuadrNotId}-i we get $N_{\sigma} \sim_{J} [x_{4}^{d}, x_{1}^{3d}, x_{2}^{-2d}, x_{3}^{4d}]$, as desired.
	\vspace{3pt} \item $d+a=c$, $d = b$, and $a+b=c+d$: we get $(a, b, c, d) = (a, 0, a, 0)$, against \eqref{CondABCD2}-iii.
	\vspace{-10pt} \item $d+a=c$, $d = b$, and $a+b=d$: we get $(a, b, c, d) = (0, d, d, d)$, that makes $N_{\sigma}$ an identity (since $d = c$).
	\vspace{3pt} \item $d+a=c$, $d+b=c$ (any choice for step 3): we fall in case \eqref{FirstCaseIV}.
\end{enumerate}

\vspace{3pt} V. $N_{\sigma} = [x_{4}^{d}, x_{3}^{c}, x_{2}^{b}, x_{1}^{a}]$. We deal with this case in four steps.

\vspace{3pt} \emph{Step 1:} If $[x_{4}^{d}, x_{1}^{b}, x_{3}^{c}]$ is not an identity, then by lemma \ref{LemmaQuadrNotId}-i we fall in case III. Otherwise, we have $d = b$ or $d+b = c$.

\vspace{3pt} \emph{Step 2:} If $[x_{4}^{d}, x_{3}^{c}, x_{1}^{a}, x_{2}^{b}]$ is not an identity, then by lemma \ref{LemmaQuadrNotId}-ii we fall in case IV. Otherwise, we have $d = c$, $d+c=a$, or $d+c+a=b$. The first case makes $N_{\sigma}$ an identity.

\vspace{3pt} \emph{Step 3:} From equation \eqref{IdLie4} we get:
\begin{align}
	& [x_{4}^{d}, x_{3}^{c}, x_{2}^{b}, x_{1}^{a}] + [x_{3}^{c}, x_{4}^{d}, x_{1}^{a}, x_{2}^{b}] + [x_{1}^{a}, x_{2}^{b}, x_{3}^{c}, x_{4}^{d}] + [x_{2}^{b}, x_{1}^{a}, x_{4}^{d}, x_{3}^{c}] = 0 \nonumber \\
	& [x_{4}^{d}, x_{3}^{c}, x_{2}^{b}, x_{1}^{a}] - [x_{4}^{d}, x_{3}^{c}, x_{1}^{a}, x_{2}^{b}] = [x_{1}^{a}, x_{2}^{b}, x_{4}^{d}, x_{3}^{c}] - [x_{1}^{a}, x_{2}^{b}, x_{3}^{c}, x_{4}^{d}] \label{EqStep3V}
\end{align}
We suppose that the two sides $A = B = [[x_{4}^{d}, x_{3}^{c}], [x_{2}^{b}, x_{1}^{a}]]$ of \eqref{EqStep3V} are not identities and $[x_{1}^{a}, x_{2}^{b}, x_{4}^{d}, x_{3}^{c}]$ is not an identity either. In this case, since $N_{\sigma}$ is not an identity by hypothesis, we have $A \sim_{J} N_{\sigma}$ by lemma \ref{LemmaQuadrNotId}-ii applied to $[x_{4}^{d}, x_{3}^{c}, x_{1}^{a}, x_{2}^{b}]$ (if the latter is not an identity), hence $N_{\sigma} \sim_{J} B \sim_{J} [x_{1}^{a}, x_{2}^{b}, x_{4}^{d}, x_{3}^{c}]$ again by lemma \ref{LemmaQuadrNotId}-ii. By lemma \ref{LemmaQuadrNotId}-i, we have $N_{\sigma} \sim_{J} [x_{1}^{a}, x_{4}^{d}, x_{2}^{b}, x_{3}^{c}] = -[x_{4}^{d}, x_{1}^{a}, x_{2}^{b}, x_{3}^{c}]$, as required, since $a = d$ and $a+d=b$ are forbidden by \eqref{CondABCD2}-(i, ii). The hypotheses fail if $a=b$, $c=d$, $a+b=c+d$, or $a+b=d$ (since $a+b+d=c$ is forbidden by \eqref{CondABCD2}-ii), but $c = d$ makes $N_{\sigma}$ an identity.

\vspace{3pt} \emph{Step 4:} Let us see when none of the previous steps applies. Failure of step 1 is equivalent to $d = b$ or $d+b = c$; failure of step 2 is equivalent to $d+c=a$ or $d+c+a=b$; failure of step 3 is equivalent to $a=b$, $a+b=c+d$, or $a+b=d$. We analyse these cases.
\begin{enumerate}
	\item If $d = b$, then it is enough to consider step 3. In fact, if $a = b$, then we get $a = d$, against \eqref{CondABCD2}-i. If $a+b=c+d$, then we get $a=c$, hence $N_{\sigma} = [x_{4}^{d}, x_{3}^{a}, x_{2}^{d}, x_{1}^{a}]$. By applying iteratively lemma \ref{LemmaQuadrNotId}, we get $N_{\sigma} \sim_{J} [x_{4}^{d}, x_{3}^{a}, x_{1}^{a}, x_{2}^{d}] \sim_{J} [x_{4}^{d}, x_{1}^{a}, x_{3}^{a}, x_{2}^{d}] \sim_{J} [x_{4}^{d}, x_{1}^{a}, x_{2}^{d}, x_{3}^{a}]$, as required, unless $d = a$, that makes $N_{\sigma}$ an identity, $d+a=a$---that is, $d=0$---against \eqref{CondABCD2}-iii, or $d+a+a=d$---that is, $a=0$---against \eqref{CondABCD2}-ii. Lastly, if $a+b=d$, then we get $a=0$, hence $a+d=b$, against \eqref{CondABCD2}-ii.
	\item If $a = b$, then it is enough to consider step 2. In fact, if $d+c=a$, then $d+c=b$, hence $N_{\sigma}$ is an identity. If $d+c+a=b$, then $d+c+b=a$, hence $N_{\sigma}$ is an identity.
	\item If $d+b=c$, $d+c=a$, and $a+b=c+d$, then we get $(a, b, c, d) = (2d, 0, d, d)$, that makes $N_{\sigma}$ an identity (since $c=d$).
	\item If $d+b=c$, $d+c+a=b$, and $a+b=c+d$, then we get $(a, b, c, d) = (0, b, b, 0)$, against \eqref{CondABCD2}-i.
	\item If $d+b=c$, $d+c=a$, and $a+b=d$, then we get $(a, b, c, d) = (3c, -c, c, 2c)$. We have $N_{\sigma} = [x_{4}^{2c}, x_{3}^{c}, x_{2}^{-c}, x_{1}^{3c}]$. By applying the Jacobi identity to $[x_{4}^{2c}, x_{3}^{c}, x_{2}^{-c}]$, we get $N_{\sigma} = -[x_{3}^{c}, x_{2}^{-c}, x_{4}^{2c}, x_{1}^{3c}]$. By applying iteratively lemma \ref{LemmaQuadrNotId}, we get $N_{\sigma} \sim_{J} [x_{3}^{c}, x_{2}^{-c}, x_{1}^{3c}, x_{4}^{2c}] \sim_{J} [x_{3}^{c}, x_{1}^{3c}, x_{2}^{-c}, x_{4}^{2c}] = -[x_{1}^{3c}, x_{3}^{c}, x_{2}^{-c}, x_{4}^{2c}] \sim_{J} [x_{1}^{3c}, x_{3}^{c}, x_{4}^{2c}, x_{2}^{-c}] \sim_{J} [x_{1}^{3c}, x_{4}^{2c}, x_{3}^{c}, x_{2}^{-c}] = -[x_{4}^{2c}, x_{1}^{3c}, x_{3}^{c}, x_{2}^{-c}] \sim_{J} [x_{4}^{2c}, x_{1}^{3c}, x_{2}^{-c}, x_{3}^{c}]$, as required.
	\item If $d+b=c$, $d+c+a=b$, and $a+b=d$, then we get $(a, b, c, d) = (-2d, 3d, 4d, d)$, hence $N_{\sigma} = [x_{4}^{d}, x_{3}^{4d}, x_{2}^{3d}, x_{1}^{-2d}]$. By applying the Jacobi identity to $[x_{4}^{d}, x_{3}^{4d}, x_{2}^{3d}]$, we get $N_{\sigma} = -[x_{3}^{4d}, x_{2}^{3d}, x_{4}^{d}, x_{1}^{-2d}] = [x_{2}^{3d}, x_{3}^{4d}, x_{4}^{d}, x_{1}^{-2d}]$. From identity \eqref{IdLie4} we get
		\[[x_{2}^{3d}, x_{3}^{4d}, x_{4}^{d}, x_{1}^{-2d}] + [x_{3}^{4d}, x_{2}^{3d}, x_{1}^{-2d}, x_{4}^{d}] = -[x_{1}^{-2d}, x_{4}^{d}, x_{3}^{4d}, x_{2}^{3d}] - [x_{4}^{d}, x_{1}^{-2d}, x_{2}^{3d}, x_{3}^{4d}].
	\]
	The two sides are not identities (since $3d \neq 4d$, $d \neq -2d$, and $3d+4d \neq d-2d$). In the l.h.s., we have $[x_{3}^{4d}, x_{2}^{3d}, x_{1}^{-2d}, x_{4}^{d}] = -[x_{2}^{3d}, x_{3}^{4d}, x_{1}^{-2d}, x_{4}^{d}] \sim_{J} [x_{2}^{3d}, x_{3}^{4d}, x_{4}^{d}, x_{1}^{-2d}]$ by lemma \ref{LemmaQuadrNotId}-ii, thus $\textnormal{l.h.s.} \sim_{J} N_{\sigma}$ by the computation above. In the r.h.s., the first term vanishes (since $-2d+d+4d=3d$), hence we have $N_{\sigma} \sim_{J} [x_{4}^{d}, x_{1}^{-2d}, x_{2}^{3d}, x_{3}^{4d}]$, as desired. \qedhere
\end{enumerate}
\end{proof}

\subsection{Fourth Family}

We consider a quadruple commutator $[x_{1}^{a}, x_{2}^{b}, x_{3}^{c}, x_{4}^{d}]$. From formula \eqref{IdAlphaBeta}, by replacing $L_{a}$, $L_{b}$, and $L_{c}$ respectively with $[L_{a}, L_{b}]$, $L_{c}$, and $L_{d}$, we get:
\begin{equation}\label{IdAlphaBetaD1}
	\alpha_{1} [L_{a}, L_{b}, L_{c}, L_{d}] - \beta_{1} [L_{a}, L_{b}, L_{d}, L_{c}] = \delta_{a+b+c+d, 0} \bigl( \alpha_{1}(c-a-b)C_{d} - \beta_{1}(d-a-b)C_{c}\bigr) \hat{c}.
\end{equation}
Similarly, by replacing $L_{a}$, $L_{b}$, and $L_{c}$ respectively with $[L_{a}, L_{d}]$, $L_{c}$, and $L_{b}$, we get:
\begin{equation}\label{IdAlphaBetaD2}
	\alpha_{2} [L_{a}, L_{d}, L_{c}, L_{b}] - \beta_{2} [L_{a}, L_{d}, L_{b}, L_{c}] = \delta_{a+b+c+d, 0} \bigl( \alpha_{2}(c-a-d)C_{b} - \beta_{2}(b-a-d)C_{c}\bigr) \hat{c}.
\end{equation}
Obviously, if $a=b$ or $a=d$, then we get identities that follow from \eqref{PolId1}; moreover, if $a+b+c+d \neq 0$, then we get identities that follow from \eqref{PolId2}. Hence, we assume $a \neq b$, $a \neq d$, and $a+b+c+d = 0$. In this case, $(a+b, c, d)$ and $(a+d, c, b)$ are weak and they have to be regular, otherwise we get identities that follow from \eqref{PolId1} by proposition \ref{PropTriple}-3. For these reasons, we give the following definition.
\begin{Def} A quadruple $(a, b, c, d) \in \Z^{4}$ is called \emph{weakly regular quadruple} if $a \neq b$, $a \neq d$, $a+b+c+d = 0$ and the triples $(a+b, c, d)$ and $(a+d, c, b)$ are regular.
\end{Def}
For example, $(1, 2, 4, -7)$ is a weakly regular quadruple. In this case, we set $k_{1} := \alpha_{1}(c-a-b)C_{d} - \beta_{1}(d-a-b)C_{c}$ and $k_{2} := \alpha_{2}(c-a-d)C_{b} - \beta_{2}(d-a-b)C_{c}$. From equations \eqref{IdAlphaBetaD1} and \eqref{IdAlphaBetaD2}, we get the following identity:
	\[k_{2}\bigl(\alpha_{1} [x_{1}^{a}, x_{2}^{b}, x_{3}^{c}, x_{4}^{d}] - \beta_{1} [x_{1}^{a}, x_{2}^{b}, x_{4}^{d}, x_{3}^{c}]\bigr) - k_{1}\bigl(\alpha_{2} [x_{1}^{a}, x_{4}^{d}, x_{3}^{c}, x_{2}^{b}] - \beta_{2} [x_{1}^{a}, x_{4}^{d}, x_{2}^{b}, x_{3}^{c}]\bigr)
\]
Moreover, $\alpha_{3} [x_{1}^{a}, x_{2}^{b}, x_{4}^{d}, x_{3}^{c}] - \beta_{3} [x_{1}^{a}, x_{4}^{d}, x_{2}^{b}, x_{3}^{c}] \in J$ because of formula \eqref{PolId3}, and $\alpha_{3}, \beta_{3} \neq 0$ since $(a, b, c, d)$ is weakly regular. Hence, we set $h_{1} := k_{2}\alpha_{1}$, $h_{2} := k_{1}\beta_{2}\frac{\alpha_{3}}{\beta_{3}}-k_{2}\beta_{1}$, and $h_{3} := -k_{1}\alpha_{2}$, and we get the identity
\begin{equation}\label{PolId4}
	\textcolor{blue}{h_{1} [x_{1}^{a}, x_{2}^{b}, x_{3}^{c}, x_{4}^{d}] + h_{2} [x_{1}^{a}, x_{2}^{b}, x_{4}^{d}, x_{3}^{c}] + h_{3} [x_{1}^{a}, x_{4}^{d}, x_{3}^{c}, x_{2}^{b}]},
\end{equation}
where $h_{1}, h_{3} \neq 0$ and none of the commutators involved is an identity.

Up to now we got the following identities:
\begin{itemize}
	\item \eqref{PolId1} for every $a \in \Z$
	\item \eqref{PolId2} for every strong regular triple $(a, b, c)$
	\item \eqref{PolId3} for every weak regular triple $(a, b, c)$ and every $d \neq a + b + c$
	\item \eqref{PolId4} for every weakly regular quadruple $(a, b, c, d)$
\end{itemize}
Let us prove that these identities generate the $T$-ideal formed by all of the $\Z$-graded identities of $\LL$.

\begin{Not} We denote by $I$ the $T$-ideal generated by \eqref{PolId1}, \eqref{PolId2}, \eqref{PolId3}, and \eqref{PolId4}. Moreover, given an $n$-uple $\pmb{a} = (a_{1}, \ldots, a_{n}) \in \Z^{n}$, we denote by $P_{n}^{\pmb{a}}$ the space of multilinear polynomials of degree $n$ in the variables $x_{1}^{a_{1}}, \ldots, x_{n}^{a_{n}}$, and by $\bar{P}_{n}^{\pmb{a}}$ the subspace of $P_{n}^{\pmb{a}}$ formed by the graded identities of $\LL$.
\end{Not}

\begin{Prop}\label{PropMultId4} Any multilinear identity $p(x_{1}^{a}, x_{2}^{b}, x_{3}^{c}, x_{4}^{d})$ belongs to $I$.
\end{Prop}
\begin{proof} Because of proposition \ref{PropMultId4Sp}, we can suppose
\begin{equation}\label{EqSumZero}\tag{$*$}
	a+b+c+d=0.
\end{equation}
We set $\pmb{a} := (a, b, c, d)$. We have to prove that $\bar{P}_{4}^{\pmb{a}} \subset I$, or, equivalently, $\bar{P}_{4}^{\pmb{a}} = \bar{P}_{4}^{\pmb{a}} \cap I$. This means that the dimension of the vector space $V := \frac{\bar{P}_{4}^{\pmb{a}}}{\bar{P}_{4}^{\pmb{a}} \cap I}$ has to be $0$. We also set $W := \frac{P_{4}^{\pmb{a}}}{P_{4}^{\pmb{a}} \cap I}$, so that $V \subset W$ (since the denominators coincide, because the elements of $I$ are identities).

\vspace{3pt} \emph{Step 1:} $\dim W \leq 3$. In fact, $P_{4}^{\pmb{a}}$ is generated by monomials of the form $N_{\sigma} := [x_{4}^{d}, x_{\sigma(1)}^{\sigma(a)}, x_{\sigma(2)}^{\sigma(b)}, x_{\sigma(3)}^{\sigma(c)}]$, where $\sigma \in S_{3}$, hence $\dim W \leq 6$. If $[x_{4}^{d}, x_{\sigma(1)}^{\sigma(a)}, x_{\sigma(2)}^{\sigma(b)}, x_{\sigma(3)}^{\sigma(c)}]$ and $[x_{4}^{d}, x_{\sigma(2)}^{\sigma(b)}, x_{\sigma(1)}^{\sigma(a)}, x_{\sigma(3)}^{\sigma(c)}]$ are not identities, then they are equivalent modulo $I$ by lemma \ref{LemmaQuadrNotId}-i; therefore, only one of the two can belong to a basis of $W$. If one of the two is an identity, then it belongs to $I$ by proposition \ref{PropMultiCom}, hence only the other one can belong to a basis of $W$. It follows that a basis of $W$ is formed at most by the commutators corresponding to $(d, a, b, c)$ or $(d, b, a, c)$, $(d, a, c, b)$ or $(d, c, a, b)$, and $(d, b, c, a)$ or $(d, c, b, a)$; therefore, we have at most $3$ elements.

\vspace{3pt} \emph{Step 2:} $\dim W \leq 2$. Let us suppose by contradiction that $\dim W = 3$. Up to exchange the variables, we assume that $[x_{4}^{d}, x_{1}^{a}, x_{2}^{b}, x_{3}^{c}]$ is not an identity. It is not restrictive to suppose that a basis of $V$ is formed by
	\[\alpha := [x_{4}^{d}, x_{1}^{a}, x_{2}^{b}, x_{3}^{c}] \qquad \beta := [x_{4}^{d}, x_{1}^{a}, x_{3}^{c}, x_{2}^{b}] \qquad \gamma := [x_{4}^{d}, x_{3}^{c}, x_{2}^{b}, x_{1}^{a}].
\]
Indeed, up to a permutation of $(a, b, c)$, let us suppose that there are two elements of the form $[x_{4}^{d}, x_{1}^{a}, \ldots]$. In this case, they are necessarily $\alpha$ and $\beta$. Up to exchange $b$ and $c$, we can suppose that the third element is of the form $[x_{4}^{d}, x_{3}^{c}, \ldots]$. The latter is necessarily $\gamma$, since, if $[x_{4}^{d}, x_{3}^{c}, x_{1}^{a}, x_{2}^{b}]$ is not an identity, then it is equivalent to $\beta$ by lemma \ref{LemmaQuadrNotId}-i. It remains to verity that taking two elements of the form $[x_{4}^{d}, x_{1}^{a}, \ldots]$ is not restrictive. In fact, if the three elements are of the form $\alpha := [x_{4}^{d}, x_{1}^{a}, x_{2}^{b}, x_{3}^{c}]$, $\mu := [x_{4}^{d}, x_{2}^{b}, \ldots]$, and $\nu := [x_{4}^{d}, x_{3}^{c}, \ldots]$ (up to exchange $b$ and $c$), then necessarily $\mu = [x_{4}^{d}, x_{2}^{b}, x_{3}^{c}, x_{4}^{a}]$ (otherwise, it would be equivalent to $\alpha$ or an identity), thus necessarily $\nu = [x_{4}^{d}, x_{3}^{c}, x_{1}^{a}, x_{2}^{b}]$ (otherwise, it would be equivalent to $\mu$ or an identity). If we can apply lemma \ref{LemmaQuadrNotId}-i to one of the three, then we get a repeated index in the second position, hence we fall in the case above. Otherwise, $(d, b, a, c)$, $(d, c, b, a)$ and $(d, a, c, b)$ correspond to identities. About the first one, if $d = b$, then $\mu$ is an identity, and, if $d+b+a=c$, then $\alpha$ is an identity; thus, only $d+b=a$ is possible. Similarly, we get $d+c=b$ and $d+a=c$. Together with $a+b+c+d=0$, we get a linear system whose only solution is the trivial one, a contradiction.

Now we can consider the basis formed by $\alpha$, $\beta$, and $\gamma$. We first prove that at least one of the following quadruples is weakly regular:
\begin{itemize}
	\item[(i)] $(d, a, b, c)$
	\item[(ii)] $(d, a, c, b)$
	\item[(iii)] $(a, d, c, b)$, with $d+c=a$
\end{itemize}
In fact, about (i), $d \neq a$ and $d \neq c$ because $\alpha$ and $\gamma$ are not identities. Let see when $(d+a, b, c)$ and $(d+c, b, a)$ are regular. If $(d+a, b, c)$ were special, then $\alpha \sim_{I} \beta$ by lemma \ref{LemmaTripleNotId}, a contradiction. If $(d+c, b, a)$ is special and $\delta := [x_{4}^{d}, x_{3}^{c}, x_{1}^{a}, x_{2}^{b}]$ is not an identity, then $\gamma \sim_{I} \delta$ by lemma \ref{LemmaTripleNotId} and $\delta \sim_{I} \beta$ by lemma \ref{LemmaQuadrNotId}-i, a contradiction. If $\delta$ is an identity, then $d+c=a$ or $d+c+a=b$. The latter case makes $\beta$ an identity, hence only the former is possible. Therefore, assuming $d+c=a$, let us see when (ii) is weakly regular. Since $d+c=a$ and $a+b+c+d=0$, we have
\begin{equation}\label{QuadrupleII}\tag{$\#$}
	(a, b, c, d) = (a, -2a, a-d, d).
\end{equation}
We already know that $d \neq a$ and $(d+a, b, c)$ is regular, hence we have to establish when $d \neq b$ and $(d+b, a, c)$ is regular. If $d = b$, then we get $(a, b, c, d) = (a, -2a, 3a, -2a)$, and the reader can check that we fall in case (iii). Assuming $d \neq b$, from \eqref{QuadrupleII} we get $(d+b, a, c) = (d-2a, a, a-d)$. There are the following possibilities for the latter to be special: (1) $d-2a=a$---that is, $d = 3a$---that would make $(d+a, b, c) = (4a, -2a, -2a)$ special; (2) $d-2a=a-d$---that is, $2d=3a$---leading to $(a, b, c, d) = (-2c, 4c, c, -3c)$; the reader can verify that we fall in case (iii); (3) $a=a-d$---that is, $d = 0$---making $\beta$ an identity, since $d+a=c$; (4) $(d-2a)+a=a-d$---that is, $a=d$---making $\alpha$ and $\beta$ identities; (5) $(d-2a)+(a-d)=a$---that is, $a=0$---making $\gamma$ an identity, since $d+c=b$; and (6) $a+(a-d)=d-2a$---that is, $d=2a$---so that $(a, b, c, d) = (a, -2a, -a, 2a)$. In the latter case, we have $\gamma = [x_{4}^{2a}, x_{3}^{-a}, x_{2}^{-2a}, x_{1}^{a}] \sim_{I} [x_{4}^{2a}, x_{2}^{-2a}, x_{3}^{-a}, x_{1}^{a}] \sim_{I} [x_{4}^{2a}, x_{2}^{-2a}, x_{1}^{a}, x_{3}^{-a}] \sim_{I} [x_{4}^{2a}, x_{1}^{a}, x_{2}^{-2a}, x_{3}^{-a}] = \alpha$, the first and the third equivalences being due to lemma \ref{LemmaQuadrNotId}-i and the second one to lemma \ref{LemmaTripleNotId}, since $(0, -a, a)$ is special.

We proved that (i), (ii), or (iii) is weakly regular. If (i) is, then identity \eqref{PolId4}, with $(a, b, c, d)$ replaced by $(d, a, b, c)$, provides the linear combination $h_{1}\alpha + h_{2}\beta + h_{3}\gamma = 0$ in $V$, with $h_{1}, h_{3} \neq 0$, a contradiction. If (ii) is, then \eqref{PolId4}, with $(a, b, c, d)$ replaced by $(d, a, c, b)$, gives $h_{1} \beta + h_{2} \alpha + h_{3} [x_{4}^{d}, x_{2}^{b}, x_{3}^{c}, x_{1}^{a}] = 0$, with $h_{1}, h_{3} \neq 0$; by applying lemma \ref{LemmaQuadrNotId}-i to the last term (that is not an identity, see the line after \eqref{PolId4}), we get $h_{1} \beta + h_{2} \alpha + h_{3} \gamma = 0$, a contradiction. If (iii) is, then \eqref{PolId4}, with $(a, b, c, d)$ replaced by $(a, d, c, b)$, leads to
\begin{equation}\label{ItemiiiEq1}\tag{$\star$}
	-h_{1} \beta - h_{2} \alpha + h_{3} [x_{1}^{a}, x_{2}^{b}, x_{3}^{c}, x_{4}^{d}] = 0.
\end{equation}
We apply identity \eqref{IdLie4} to the third commutator, getting
\begin{align}
	[x_{1}^{a}, x_{2}^{b}, x_{3}^{c}, x_{4}^{d}] &= -\bigl([x_{2}^{b}, x_{1}^{a}, x_{4}^{d}, x_{3}^{c}] + [x_{4}^{d}, x_{3}^{c}, x_{2}^{b}, x_{1}^{a}] + [x_{3}^{c}, x_{4}^{d}, x_{1}^{a}, x_{2}^{b}]\bigr) \nonumber \\
	&= [x_{1}^{a}, x_{2}^{b}, x_{4}^{d}, x_{3}^{c}] - [x_{4}^{d}, x_{3}^{c}, x_{2}^{b}, x_{1}^{a}] + [x_{4}^{d}, x_{3}^{c}, x_{1}^{a}, x_{2}^{b}]. \label{ItemiiiEq2}\tag{$\star\star$}
\end{align}
The first commutator in \eqref{ItemiiiEq2} is not an identity (see the line after \eqref{PolId4}) and it is equivalent to $[x_{1}^{a}, x_{4}^{d}, x_{2}^{b}, x_{3}^{c}] = -\alpha$ by lemma \ref{LemmaQuadrNotId}-i. The second commutator is $\gamma$. The third commutator is an identity because of the hypothesis $d+c=a$. Hence, \eqref{ItemiiiEq1} becomes $-h_{1}\beta - h_{2}\alpha + h_{3}(-\alpha-\gamma) = 0$---that is, $-(h_{2}+h_{3})\alpha - h_{1}\beta - h_{3}\gamma = 0$. This is a contradiction, since $h_{1}, h_{3} \neq 0$.

\vspace{3pt} \emph{Step 3:} $\dim V = 0$. Since $V \subset W$, the thesis is trivial if $\dim W = 0$. Hence, because of the previous step, we have to analyse the two cases $\dim W = 1$ and $\dim W = 2$.

\vspace{3pt} \texttt{Case I:} $\dim W = 1$. Let us suppose by contradiction that $\dim V = 1$---that is, $V = W$ and $P_{4}^{\pmb{a}} = \bar{P}_{4}^{\pmb{a}}$---and, up to a permutation of the indices, let $\alpha := [x_{4}^{d}, x_{1}^{a}, x_{2}^{b}, x_{3}^{c}]$ be a representative of a basis of $V$. Since $\alpha \in \bar{P}_{4}^{\pmb{a}}$, it is an identity, hence it belongs to $I$ by proposition \ref{PropMultiCom}, a contradiction.

\vspace{3pt} \texttt{Case II:} $\dim W = 2$. We start by assuming that the triple $(d+a, b, c)$ is regular and $d \neq a$. We set
\begin{equation}\label{BasisAlphaBeta}
	\alpha := [x_{4}^{d}, x_{1}^{a}, x_{2}^{b}, x_{3}^{c}] \qquad \beta := [x_{4}^{d}, x_{1}^{a}, x_{3}^{c}, x_{2}^{b}]
\end{equation}
In this case, we claim that $\alpha$ and $\beta$ form a basis of $W$. In fact, otherwise, they are not linearly independent, hence there exist non-trivial $\lambda, \mu \in \K$ such that $\lambda\alpha + \mu\beta \in I$. The hypotheses stated imply that $\alpha$ and $\beta$ are not identities, thus $\lambda$ and $\mu$ are both different from $0$. Since $\lambda\alpha + \mu\beta$ is an identity, $\lambda[x_{5}^{d+a}, x_{2}^{b}, x_{3}^{c}] + \mu[x_{5}^{d+a}, x_{3}^{c}, x_{2}^{b}]$ is an identity of degree $3$, that belongs to $K$ by proposition \ref{PropMultiL}. Nevertheless, since $\lambda, \mu \neq 0$ and $(d+a, b, c)$ is weak and regular, we fall in case II of the proof of proposition \ref{PropMultiL}, forbidding that $\lambda \alpha + \mu \beta$ is an identity, a contradiction.

Since $\alpha$ and $\beta$ represent a basis of $W$, any non trivial element of $V$ is of the form $p = \lambda \alpha + \mu \beta$ (up to $I$), with $\lambda, \mu \in \K$. It follows that $q := \lambda [x_{5}^{a+d}, x_{2}^{b}, x_{3}^{c}] + \mu [x_{5}^{a+d}, x_{3}^{c}, x_{2}^{b}]$ is an identity too, and $p$ belongs to the $T$-ideal generated by $q$. Since $q \in I$ by proposition \ref{PropMultiL}, we have $p \in I$, a contradiction. Thus, $\dim V = 0$.

It remains to prove that assuming $(d+a, b, c)$ regular and $d \neq a$ is not restrictive up to a permutation of $(a, b, c, d)$. Since $\dim W = 2$ by hypothesis, there exists a commutator that is not an identity (otherwise, it would belong to $I$ by proposition \ref{PropMultiCom}). Up to a permutation of $(a, b, c, d)$, let us suppose that it is $\alpha := [x_{4}^{d}, x_{1}^{a}, x_{2}^{b}, x_{3}^{c}]$. It follows that $d \neq a$, $d+a \neq b$, and $d+a+b \neq c$. Hence, if $(d+a, b, c)$ is special, then we have the following possibilities: (1) $d+a=c$; (2) $b=c$; (3) $d+a+c=b$; and (4) $d+a=b+c$. Moreover, since $\dim W = 2$, we can find a basis of $W$ represented by $\alpha$ and $\eta = [x_{4}^{d}, \ldots]$. Let us show that, in each of the cases (1)--(4), there exists a permutation $\sigma$ such that $\sigma(d) \neq \sigma(a)$ and $(\sigma(d)+\sigma(a), \sigma(b), \sigma(c))$ is regular.

\vspace{3pt} (1) Since $a+b+c+d=0$, we get $(a, b, c, d) = (c-d, -2c, c, d)$. Let us see when $d \neq c$ and $(d+c, a, b)$ is regular. If $d = c$, then we get $(a, b, c, d) = (0, -2d, d, d)$, so that the only possibilities for $\eta$ are $[x_{4}^{d}, x_{1}^{a}, x_{3}^{c}, x_{2}^{b}]$, which is an identity (since $d+a=c$), $[x_{4}^{d}, x_{2}^{b}, x_{1}^{a}, x_{3}^{c}]$, which is equivalent to $\alpha$ by lemma \ref{LemmaQuadrNotId}-i, and $[x_{4}^{d}, x_{2}^{b}, x_{3}^{c}, x_{1}^{a}]$, which is an identity (since $d+b+c=a$). Hence, $d \neq c$. Let us show when $(d+c, a, b)$ is special. We have six possibilities: (i) $d+c=a$---that is, $d=0$; (ii) $d+c=b$---that is, $d=-3c$; (iii) $a=b$---that is, $d=3c$; (iv) $d+c+a=b$---that is, $c=0$; (v) $d+c+b=a$---that is, $d=c$; and (vi) $a+b=d+c$---that is, $d=-c$. In case (i), we get $(a, b, c, d) = (c, -2c, c, 0)$. We have $\alpha = [x_{4}^{0}, x_{1}^{c}, x_{2}^{-2c}, x_{3}^{c}]$ and the only possibilities for $\eta$ that are not identities are $[x_{4}^{0}, x_{2}^{-2c}, x_{1}^{c}, x_{3}^{c}]$, that is equivalent to $\alpha$ by lemma \ref{LemmaQuadrNotId}-i, $[x_{4}^{0}, x_{2}^{-2c}, x_{3}^{c}, x_{1}^{c}]$, that is equivalent to the previous one by lemma \ref{LemmaTripleNotId} (since $(-2c, c, c)$ is special), and $[x_{4}^{0}, x_{3}^{c}, x_{2}^{-2c}, x_{1}^{c}]$, that is equivalent to the previous one by lemma \ref{LemmaQuadrNotId}-i. In case (ii), we get $(a, b, c, d) = (4c, -2c, c, -3c)$, so that $d \neq b$ and $(d+b, a, c)$ is regular. In case (iii), we get $(a, b, c, d) = (-2c, -2c, c, 3c)$, so that $a \neq c$ and $(a+c, b, d)$ is regular. In case (iv), we get $(a, b, c, d) = (-d, 0, 0, d)$, making $\alpha$ an identity (since $d+a=b$). Case (v) has been already excluded above. In case (vi), we get $(a, b, c, d) = (-2d, 2d, -d, d)$, so that $d \neq b$ and $(d+b, a, c)$ is regular.

\vspace{3pt} (2) We get $(a, b, c, d) = (-2c-d, c, c, d)$. If $d = c$, then we get $(a, b, c, d) = (-3d, d, d, d)$, so that $\alpha = [x_{4}^{d}, x_{1}^{-3d}, x_{2}^{d}, x_{3}^{d}]$ and the only possibility for $\eta$ is $[x_{4}^{d}, x_{1}^{-3d}, x_{3}^{d}, x_{2}^{d}]$, that is equivalent to $\alpha$ by \ref{LemmaTripleNotId} (since $(-2d, d, d)$ is special). Let us show when $(d+c, a, b)$ is special. We have six possibilities: (i) $d+c=a$---that is, $2d=-3c$; (ii) $d+c=b$---that is, $d=0$; (iii) $a=b$---that is, $d=-3c$; (iv) $d+c+a=b$---that is, $c=0$; (v) $d+c+b=a$---that is, $d=-2c$; and (vi) $a+b=d+c$---that is, $d=-c$. In case (i), we get $(a, b, c, d) = (a, -2a, -2a, 3a)$, so that $a \neq b$ and $(a+b, c, d) = (-a, -2a, 3a)$ is regular. In case (ii), we get $(a, b, c, d) = (-2c, c, c, 0)$. We have $\alpha = [x_{4}^{0}, x_{1}^{-2c}, x_{2}^{c}, x_{3}^{c}]$. The only possibilities for $\eta$ that are not identities are $[x_{4}^{0}, x_{1}^{-2c}, x_{3}^{c}, x_{2}^{c}]$, that is equivalent to $\alpha$ by lemma \ref{LemmaTripleNotId}; $[x_{4}^{0}, x_{3}^{c}, x_{1}^{-2c}, x_{2}^{c}]$, that is equivalent to previous one by lemma \ref{LemmaQuadrNotId}-i; and $[x_{4}^{0}, x_{2}^{c}, x_{1}^{-2c}, x_{3}^{c}]$, that is equivalent to $\alpha$ by lemma \ref{LemmaQuadrNotId}-i. In case (iii), we have $(a, b, c, d) = (c, c, c, -3c)$, so that $\alpha = [x_{4}^{-3c}, x_{1}^{c}, x_{2}^{c}, x_{3}^{c}]$ and there are no possibilities for $\eta$. In case (iv), we have $(a, b, c, d) = (-d, 0, 0, d)$, making $\alpha$ an identity. In case (v), we have $(a, b, c, d) = (0, c, c, -2c)$, so that $\alpha = [x_{4}^{-2c}, x_{1}^{0}, x_{2}^{c}, x_{3}^{c}]$ and the only possibilities for $\eta$ that are not identities are $[x_{4}^{-2c}, x_{2}^{c}, x_{1}^{0}, x_{3}^{c}]$, that is equivalent to $\alpha$ by lemma \ref{LemmaQuadrNotId}-i; $[x_{4}^{-2c}, x_{1}^{0}, x_{3}^{c}, x_{2}^{c}]$, that is equivalent to $\alpha$ by lemma \ref{LemmaTripleNotId}; and $[x_{4}^{-2c}, x_{3}^{c}, x_{1}^{0}, x_{2}^{c}]$, that is equivalent to the previous one by lemma \ref{LemmaQuadrNotId}-i. In case (vi), we have $(a, b, c, d) = (-c, c, c, -c)$, making $\alpha$ an identity (since $d = a$).

\vspace{3pt} (3) We get $(a, b, c, d) = (-c-d, 0, c, d)$. If $d = 0$, then we get $(a, b, c, d) = (-c, 0, c, 0)$, so that $\alpha = [x_{4}^{0}, x_{1}^{-c}, x_{2}^{0}, x_{3}^{c}]$ and the only possibilities for $\eta$ are $[x_{4}^{0}, x_{1}^{-c}, x_{3}^{c}, x_{2}^{0}]$, that is an identity, $[x_{4}^{0}, x_{3}^{c}, x_{1}^{-c}, x_{2}^{0}]$, that is an identity too, and $[x_{4}^{0}, x_{3}^{c}, x_{2}^{0}, x_{1}^{-c}]$, that is equivalent to $\alpha$ by identity \eqref{IdLie4}. Hence, $d \neq 0 = b$. Let us show when $(d+b, a, c)$ is special. We have six possibilities: (i) $d+b=a$---that is, $c=-2d$; (ii) $d+b=c$---that is, $d=c$; (iii) $a=c$---that is, $d = -2c$; (iv) $d+b+a=c$---that is, $c=0$; (v) $d+b+c=a$---that is, $d=-c$; and (vi) $d+b=a+c$---that is, $d=0$. In case (i), we have $(a, b, c, d) = (d, 0, -2d, d)$, making $\alpha$ an identity (since $d=a$). In case (ii), we have $(a, b, c, d) = (-2d, 0, d, d)$, so that $\alpha = [x_{4}^{d}, x_{1}^{-2d}, x_{2}^{0}, x_{3}^{d}]$ and the only possibilities for $\eta$ are $[x_{4}^{d}, x_{1}^{-2d}, x_{3}^{d}, x_{2}^{0}]$, that is an identity, $[x_{4}^{d}, x_{2}^{0}, x_{1}^{-2d}, x_{3}^{d}]$, that is equivalent to $\alpha$ by lemma \ref{LemmaQuadrNotId}-i, and $[x_{4}^{d}, x_{2}^{0}, x_{3}^{d}, x_{1}^{-2d}]$, that is an identity (since $d+0=d$). In case (iii), we have $(a, b, c, d) = (c, 0, c, -2c)$, so that $\alpha = [x_{4}^{-2c}, x_{1}^{c}, x_{2}^{0}, x_{3}^{c}]$ and the only possibilities for $\eta$ are $[x_{4}^{-2c}, x_{1}^{c}, x_{3}^{c}, x_{2}^{0}]$, that is an identity, and $[x_{4}^{-2c}, x_{2}^{0}, x_{3}^{c}, x_{1}^{c}]$, that is equivalent to $\alpha$ by lemma \ref{LemmaQuadrNotId}-i. In case (iv), we have $(a, b, c, d) = (-d, 0, 0, d)$, making $\alpha$ an identity (since $d+a=c$). In case (v), we have $(a, b, c, d) = (0, 0, -d, d)$, so that $\alpha = [x_{4}^{d}, x_{1}^{0}, x_{2}^{0}, x_{3}^{-d}]$ and the only possibilities for $\eta$ are $[x_{4}^{d}, x_{1}^{0}, x_{3}^{-d}, x_{2}^{0}]$ and $[x_{4}^{d}, x_{3}^{-d}, x_{1}^{0}, x_{2}^{0}]$, that are identities. Case (vi) has already been excluded.

\vspace{3pt} (4) We get $(a, b, c, d) = (-d, -c, c, d)$. If $d = c$, then we have $(a, b, c, d) = (-d, -d, d, d)$, so that $\alpha = [x_{4}^{d}, x_{1}^{-d}, x_{2}^{-d}, x_{3}^{d}]$ and the only possibility for $\eta$ is $[x_{4}^{d}, x_{1}^{-d}, x_{3}^{d}, x_{2}^{-d}]$, that is equivalent to $\alpha$ by lemma \ref{LemmaQuadrNotId}-ii (since $(0, d, -d)$ is special). Let us show when $(d+c, a, b)$ is special. We have six possibilities: (i) $d+c=a$---that is, $c=-2d$; (ii) $d+c=b$---that is, $d=-2c$; (iii) $a=b$---that is, $d=c$; (iv) $d+c+a=b$---that is, $c=0$; (v) $d+c+b=a$---that is, $d=0$; and (vi) $a+b=d+c$---that is, $d=-c$. In case (i), we get $(a, b, c, d) = (-d, 2d, -2d, d)$, so that $d \neq b$ and $(d+b, a, c) = (3d, -d, -2d)$ is regular. In case (ii), we get $(a, b, c, d) = (2c, -c, c, -2c)$, so that $a \neq c$ and $(a+c, b, d) = (3c, -c, -2c)$ is regular. Case (iii) has already been excluded. Case (iv) coincides with case (3)-(iv). In case (v), we have $(a, b, c, d) = (0, -c, c, 0)$, making $\alpha$ an identity (since $a = d$). In case (vi), we have $(a, b, c, d) = (-d, d, -d, d)$, so that $\alpha = [x_{4}^{d}, x_{1}^{-d}, x_{2}^{d}, x_{3}^{-d}]$ and the only possibility for $\eta$ is $[x_{4}^{d}, x_{1}^{-d}, x_{3}^{-d}, x_{2}^{d}]$, that is equivalent to $\alpha$ by lemma \ref{LemmaQuadrNotId}-ii (since $(0, -d, d)$ is special).
\end{proof}

The extension of proposition \ref{PropMultId4} to identities of any degree can be found in \cite[Lemma 3.11--Theorem 3.19]{FDK}.


\section{Infinite Field of Positive Characteristic}

Now we assume that $\K$ is an infinite field of characteristic $p \neq 2, 3$.

\begin{Not*} From now on, given $n \in \Z$, we denote by $\bar{n}$ its projection in $\Z_{p}$.
\end{Not*}

The Virasoro algebra is still defined by \eqref{StructVirasoro}, assuming $K \in \Z_{p}^{*}$ and projecting the integral coefficients $m-n$ and $m(m^{2}-1)$ to $\Z_{p}$ (while $\delta_{m+n, 0} = 1$ if and only if $m+n = 0$ in $\Z$, not in $\Z_{p}$). The restriction on the characteristic is due to the fact that, if $p = 2$ or $p = 3$, then the Virasoro algebra coincides with $U_{1}$, since $C_{m}$ vanishes for every $m$.

\subsection{Multihomogeneous \emph{vs} Multilinear Identities}

In general, when the characteristic of the underlying field is positive, the T-ideal formed by the (graded) polynomial identities of a fixed algebra is generated by its multihomogeneous elements, not necessarily by its multilinear ones. Nevertheless, dealing with the Virasoro algebra, the multilinear identities generate the corresponding T-ideal. Essentially, this is due to the fact that the vector subspace of $\LL$ formed by the elements of any fixed degree is one-dimensional up to the centre. We obtain the following proposition.

\begin{Prop} The T-ideal $T_{\Z}(\LL)$ is generated by its multilinear elements.
\end{Prop}

The proof can be found in \cite[Lemma 4.1]{FDK}.

\subsection{First Family}

We have the following identities:
\begin{equation}\label{PolId1CarP}
	\textcolor{blue}{[x_{1}^{a}, x_{2}^{b}]}, \quad \bar{a} = \bar{b}.
\end{equation}
That's because, if $\bar{m} = \bar{n}$, then $[L_{m}, L_{n}] = 0$ (and the generator $\hat{c}$ is central). In fact, if $n = m + kp$, then from \eqref{StructVirasoro} we get $[L_{m}, L_{m+kp}] = \delta_{2m+kp, 0} C_{\bar{m}} \hat{c}$. If $\delta_{2m+kp, 0} = 1$, then $p \mid 2m$, thus $p \mid m$, that is, $\bar{m} = 0$. This implies that $C_{\bar{m}} = 0$, thus $[L_{m}, L_{n}] = 0$. The converse holds too, since, if $[L_{m}, L_{n}] = 0$, then from \eqref{StructVirasoro} we get $\bar{m} - \bar{n} = 0$.

\begin{Prop}\label{PropMultiComCarP} If $[x_{1}^{a_{1}}, \ldots, x_{n}^{a_{n}}]$ is an identity of $\LL$, then it belongs to the $T$-ideal generated by the identities \eqref{PolId1CarP}.
\end{Prop}
\begin{proof} Analogous to the proof of proposition \ref{PropMultiCom}, considering that the case $n=2$ was proven above.
\end{proof}

\subsection{Second Family}

Equations \eqref{EqLABC}--\eqref{EqLBCA} remain unchanged, projecting the coefficients to $\Z_{p}$. Lemma \ref{LemmaAlphaBetaZero} (with $\bar{\alpha}$ and $\bar{\beta}$ instead of $\alpha$ and $\beta$) and equation \eqref{IdAlphaBeta} do not change either. Hence, under the condition $a + b + c \neq 0$, the term $\delta_{a+b+c, 0}$ in \eqref{IdAlphaBeta} leads to the following identity:
\begin{equation}\label{PolId2CarP}
	\textcolor{blue}{\bar{\alpha} [x_{1}^{a}, x_{2}^{b}, x_{3}^{c}] - \bar{\beta} [x_{1}^{a}, x_{3}^{c}, x_{2}^{b}]}
\end{equation}

We already defined strong and weak triples in \ref{DefPiTr}. Now we add the following definitions.

\begin{Def}\label{DefPiTrCarP} A triple $(a, b, c) \in \Z^{3}$ is called \emph{bar-regular} if the three entries of $(\bar{a}, \bar{b}, \bar{c})$ are pairwise distinct and none of the three is the sum of the other two; otherwise, it is called \emph{bar-special}.
\end{Def}

The polynomial \eqref{PolId2CarP} is a non-trivial identity only for strong bar-regular triples; otherwise, it is not an identity or it is a multiple of one of the three commutators in \eqref{ThreeComm}, as the following adaptation of proposition \ref{PropTriple} shows (with essentially the same proof).

\begin{Prop}\label{PropTripleCarP} A triple $(a, b, c) \in \Z^{3}$ is:
\begin{enumerate}
	\item strong and bar-regular if and only if formula \eqref{PolId2CarP} is a non-vanishing identity and it is not a multiple of a commutator in \eqref{ThreeComm}; in this case, none of such commutators is an identity;
	\item weak and bar-regular if and only if formula \eqref{PolId2CarP} is not an identity; in this case, none of the commutators in \eqref{ThreeComm} is an identity either; and
	\item bar-special if and only if formula \eqref{PolId2CarP} is an identity and it is a multiple of a commutator in \eqref{ThreeComm}; in this case, at least one of the commutators in \eqref{ThreeComm} is an identity (even if formula \eqref{PolId2CarP} vanishes).
\end{enumerate}
Moreover:
\begin{itemize}
	\item[\emph{(ii)}] In case (3), $(a, b, c)$ is weak if and only if there exist $k, h \in \Z$ such that $(a, b, c)$ coincides with one of the following triples: $(k, k+hp, -2k-hp)$, $(k, -2k-hp, k+hp)$, $(-2k-hp, k, k+hp)$, $(k, -k-hp, hp)$, $(k, hp, -k-hp)$, $(hp, k, -k-hp)$.
	\item[\emph{(iii)}] Formula \eqref{PolId2CarP} vanishes if and only if there exist $k, h, l \in \Z$ such that $(a, b, c)$ takes one of the following forms: $(k, k+hp, k+lp)$, $(k, k+hp, lp)$, $(k, lp, k+hp)$, $(lp, k, k+hp)$.
\end{itemize}
\end{Prop}

\begin{Not} We denote by $K$ the $T$-ideal generated by the identities \eqref{PolId1CarP} and \eqref{PolId2CarP}, assuming $(a, b, c)$ strong and bar-regular in the latter case.
\end{Not}

Lemma \ref{LemmaTripleNotId} and proposition \ref{PropMultiL} remain essentially unchanged. 

\subsection{Third Family}

Formula \eqref{IdAlphaBeta} shows that \eqref{PolId2CarP} is central, hence we get the following identity for every $d \in \Z$:
\begin{equation}\label{PolId3CarP}
	\textcolor{blue}{\bar{\alpha} [x_{1}^{a}, x_{2}^{b}, x_{3}^{c}, x_{4}^{d}] - \bar{\beta} [x_{1}^{a}, x_{3}^{c}, x_{2}^{b}, x_{4}^{d}]}.
\end{equation}
This identity is non-trivial when \eqref{PolId2CarP} is \emph{not} an identity---that is, when $(a, b, c)$ is weak and bar-regular. Moreover, if $\bar{d} = \bar{a} + \bar{b} + \bar{c}$, then the two quadruple commutators are both identities following from \eqref{PolId1CarP}, hence \eqref{PolId3CarP} is trivial.

\begin{Not} We denote by $J$ the $T$-ideal generated by the identities \eqref{PolId1CarP}, \eqref{PolId2CarP} with $(a, b, c)$ strong and bar-regular, and \eqref{PolId3CarP} with $(a, b, c)$ weak and bar-regular and $\bar{d} \neq \bar{a}+\bar{b}+\bar{c}$.
\end{Not}

Lemma \ref{LemmaQuadrNotId} holds in this framework too, with no essential variations (we stress that, in item ii, the hypothesis is still $a + b + c + d \neq 0$, without projecting it to $\Z_{p}$). About proposition \ref{PropMultId4Sp}, the structure of the proof remains unchanged, applying conditions \eqref{CondABCD1} and \eqref{CondABCD2} to $(\bar{a}, \bar{b}, \bar{c}, \bar{d})$. About \eqref{CondABCD1}, if each index coincides with the sum of the others, then $(\bar{a}, \bar{b}, \bar{c}, \bar{d}) = (0, 0, 0, 0)$, that is, $(a, b, c, d) = (k_{1}p, k_{2}p, k_{3}p, k_{4}p)$, so that proposition \ref{PropMultiComCarP} applies to every commutator in the variables $x_{1}^{a}, x_{2}^{b}, x_{3}^{c}, x_{4}^{d}$. The rest of the proof holds without variations if $p > 7$. In the case $p = 5$, the following items fail: III-Step 3-(2), because $[x_{4}^{3c}, x_{3}^{c}, x_{2}^{-c}, x_{1}^{2c}]$ is an identity (since $3c+c=-c$); III-Step 3-(6), because $[x_{1}^{a}, x_{3}^{3a}, x_{2}^{-a}, x_{4}^{2a}]$ is an identity (since $a+3a=-a$); IV-Step 4-(3), because $[x_{2}^{3c}, x_{1}^{c}, x_{3}^{-c}, x_{4}^{2c}]$ is an identity (since $3c+c=-c$);\footnote{In IV-Step 4-(5), $N_{\sigma}$ is an identity (since $d+4d+3d=-2d$), hence proposition \ref{PropMultiComCarP} applies.} and V-Step 4-(5), because $[x_{3}^{c}, x_{1}^{3c}, x_{2}^{-c}, x_{4}^{2c}]$ is an identity (since $c+3c=-c$). For these reasons, we have to add the following family:
\begin{equation}\label{PolId3Car5}
	\textcolor{blue}{[x_{1}^{a}, x_{2}^{b}, x_{3}^{c}, x_{4}^{d}] + 2 [x_{1}^{a}, x_{4}^{d}, x_{2}^{b}, x_{3}^{c}]}, \quad p = 5, \; (\bar{a}, \bar{b}, \bar{c}, \bar{d}) = (\bar{a}, -2\bar{a}, 2\bar{a}, -\bar{a}),
\end{equation}
assuming $\bar{a} \neq 0$ and $a+b+c+d \neq 0$. The proof that it is actually an identity can be realised by direct computation. Let us show that \eqref{PolId3Car5} settles the issues raised in the previous paragraph:
\begin{itemize}
	\item III-Step 3-(2): $N_{\sigma} = [x_{4}^{3c}, x_{2}^{-c}, x_{3}^{c}, x_{1}^{2c}]$. We apply \eqref{PolId3Car5} with $(a, b, c, d) = (3c, -c, c, 2c)$ and we get $N_{\sigma} \sim_{J} [x_{4}^{3c}, x_{1}^{2c}, x_{2}^{-c}, x_{3}^{c}]$, as required.
	\item III-Step 3-(6): $N_{\sigma} = [x_{4}^{2a}, x_{2}^{-a}, x_{3}^{3a}, x_{1}^{a}]$. By applying lemma \ref{LemmaQuadrNotId}, we get $N_{\sigma} \sim_{J} [x_{4}^{2a}, x_{3}^{3a}, x_{2}^{-a}, x_{1}^{a}] \sim_{J} [x_{4}^{2a}, x_{3}^{3a}, x_{1}^{a}, x_{2}^{-a}]$. We apply \eqref{PolId3Car5} with $(a, b, c, d)$ replaced by $(2a, a, -a, 3a)$ and we get $N_{\sigma} \sim_{J} [x_{4}^{2a}, x_{1}^{a}, x_{2}^{-a}, x_{3}^{3a}]$, as required.
	\item IV-Step 4-(3): $N_{\sigma} = [x_{4}^{2c}, x_{3}^{c}, x_{1}^{-c}, x_{2}^{3c}]$. We apply \eqref{PolId3Car5} with $(a, b, c, d) = (2c, c, -c, 3c)$ and we get $N_{\sigma} \sim_{J} [x_{4}^{2c}, x_{1}^{-c}, x_{2}^{3c}, x_{3}^{c}]$, as required.
	\item V-Step 4-(5): $N_{\sigma} = [x_{4}^{2c}, x_{3}^{c}, x_{2}^{-c}, x_{1}^{3c}]$. We apply \eqref{PolId3Car5} with $(a, b, c, d) = (2c, c, -c, 3c)$ and we get $N_{\sigma} \sim_{J} [x_{4}^{2c}, x_{2}^{-c}, x_{1}^{3c}, x_{3}^{c}]$. By applying lemma \ref{LemmaQuadrNotId}-i, we get $N_{\sigma} \sim_{J} [x_{4}^{2c}, x_{1}^{3c}, x_{2}^{-c}, x_{3}^{c}]$, as required.
\end{itemize}
In the remaining cases, $N_{\sigma}$ is an identity: in III-Step 3-(4), because $-2a = 3a$, and in IV-Step 4-(5) and V-Step 4-(6), because $3d+4d+d=-2d$.

Similarly, in the case $p = 7$, the following items fail: III-Step 3-(2), because $[x_{4}^{3c}, x_{3}^{c}, x_{1}^{2c}, x_{2}^{-c}]$ is an identity (since $3c+c+2c=-c$); III-Step 3-(6), because $[x_{4}^{2a}, x_{3}^{3a}, x_{1}^{a}, x_{2}^{-a}]$ is an identity; IV-Step 4-(3), because $[x_{2}^{3c}, x_{1}^{c}, x_{4}^{2c}, x_{3}^{-c}]$ is an identity; V-Step 4-(5), because $[x_{1}^{3c}, x_{3}^{c}, x_{4}^{2c}, x_{2}^{-c}]$ is an identity. For these reasons, we have to add the following family:

\vspace{-10pt}
\begin{small}
\begin{equation}\label{PolId3Car7}
	\textcolor{blue}{[x_{1}^{a}, x_{2}^{b}, x_{3}^{c}, x_{4}^{d}] - [x_{1}^{a}, x_{4}^{d}, x_{2}^{b}, x_{3}^{c}]}, \quad p = 7, \; (\bar{a}, \bar{b}, \bar{c}, \bar{d}) = (\bar{a}, 2\bar{a}, -2\bar{a}, 3\bar{a}), (\bar{a}, 3\bar{a}, -2\bar{a}, -3\bar{a}),
\end{equation}
\end{small}
\vspace{-12pt}

\noindent assuming $\bar{a} \neq 0$ and $a+b+c+d \neq 0$. Then:
\begin{itemize}
	\item III-Step 3-(2): $N_{\sigma} = [x_{4}^{3c}, x_{2}^{-c}, x_{3}^{c}, x_{1}^{2c}]$. We apply \eqref{PolId3Car7} with $(a, b, c, d) = (3c, -c, c, 2c)$, and we get $N_{\sigma} \sim_{J} [x_{4}^{3c}, x_{1}^{2c}, x_{2}^{-c}, x_{3}^{c}]$, as required.
	\item III-Step 3-(6): $N_{\sigma} = [x_{4}^{2a}, x_{2}^{-a}, x_{3}^{3a}, x_{1}^{a}]$. We apply \eqref{PolId3Car7} with $(a, b, c, d)$ replaced by $(2a, -a,$ $3a, a)$, and we get $N_{\sigma} \sim_{J} [x_{4}^{2a}, x_{1}^{a}, x_{2}^{-a}, x_{3}^{3a}]$, as required.
	\item IV-Step 4-(3): $N_{\sigma} = [x_{4}^{2c}, x_{3}^{c}, x_{1}^{-c}, x_{2}^{3c}]$. We apply \eqref{PolId3Car7} with $(a, b, c, d)$ replaced by $(2c, -c,$ $3c, c)$, and we get $N_{\sigma} \sim_{J} [x_{4}^{2c}, x_{1}^{-c}, x_{2}^{3c}, x_{3}^{c}]$, as required.
	\item V-Step 4-(5): $N_{\sigma} = [x_{4}^{2c}, x_{3}^{c}, x_{2}^{-c}, x_{1}^{3c}]$. We apply \eqref{PolId3Car7} with $(a, b, c, d)$ replaced by $(2c, -c, 3c,$ $c)$, and we get $N_{\sigma} \sim_{J} [x_{4}^{2c}, x_{2}^{-c}, x_{1}^{3c}, x_{3}^{c}]$. By applying lemma \ref{LemmaQuadrNotId}-i, we get $N_{\sigma} \sim_{J} [x_{4}^{2c}, x_{1}^{3c}, x_{2}^{-c}, x_{3}^{c}]$, as required.
\end{itemize}
Cases III-Step 3-(4), IV-Step 4-(5), and V-Step 4-(6) can be handled as for any $p > 7$.

\subsection{Fourth Family}

Formulas \eqref{IdAlphaBetaD1} and \eqref{IdAlphaBetaD2} keep on holding, thus we give the following definition.
\begin{Def} A quadruple $(a, b, c, d) \in \Z^{4}$ is called \emph{weakly bar-regular quadruple} if $\bar{a} \neq \bar{b}$, $\bar{a} \neq \bar{d}$, $a+b+c+d = 0$ and the triples $(a+b, c, d)$ and $(a+d, c, b)$ are bar-regular.
\end{Def}
In this case, we get the identity
\begin{equation}\label{PolId4CarP}
	\textcolor{blue}{\bar{h}_{1} [x_{1}^{a}, x_{2}^{b}, x_{3}^{c}, x_{4}^{d}] + \bar{h}_{2} [x_{1}^{a}, x_{2}^{b}, x_{4}^{d}, x_{3}^{c}] + \bar{h}_{3} [x_{1}^{a}, x_{4}^{d}, x_{3}^{c}, x_{2}^{b}]},
\end{equation}
where $\bar{h}_{1}, \bar{h}_{3} \neq 0$ and none of the commutators involved is an identity.

\begin{Not} We denote by $I$ the $T$-ideal generated by \eqref{PolId1CarP}, \eqref{PolId2CarP}, \eqref{PolId3CarP}, and \eqref{PolId4CarP}, adding \eqref{PolId3Car5} if $p = 5$ and \eqref{PolId3Car7} if $p = 7$.
\end{Not}

Proposition \ref{PropMultId4} holds in characteristic $p$ too, adapting the proof as follows. We still suppose that $a+b+c+d=0$ because of proposition \ref{PropMultId4Sp}. Step 1 remains unchanged. About step 2, we can still suppose that $\alpha$, $\beta$, and $\gamma$ represent a basis of $V$, since the hypothesis $a+b+c+d=0$ trivially implies $\bar{a}+\bar{b}+\bar{c}+\bar{d}=0$, that, together with the obstructions $\bar{d}+\bar{b}=\bar{a}$, $\bar{d}+\bar{c}=\bar{b}$, and $\bar{d}+\bar{a}=\bar{c}$ lead to a linear system whose only solution is the trivial one. In this case, we have $(a, b, c, d) = (k_{1}p, k_{2}p, k_{3}p, k_{4}p)$, so that every commutator in the variables $x_{1}^{a}, x_{2}^{b}, x_{3}^{c}$, and $x_{4}^{d}$ is an identity; thus, proposition \ref{PropMultiComCarP} applies. Moreover, let us show that one the quadruples (i), (ii), and (iii) is weakly bar-regular. We only have to analyse the special cases leading to (iii) in zero characteristic:
\begin{itemize}
	\item If $(a, b, c, d) = (a, -2a, 3a, -2a)$, then $(a+d, c, b) = (-a, 3a, -2a)$ and $(a+b, d, c) = (-a, -2a, 3a)$. If $p > 5$, then both are bar-regular and $\bar{a} \neq \bar{d}, \bar{b}$, thus (iii) applies; otherwise, $\bar{d} = \bar{c}$, thus $\gamma$ is an identity (equivalently, $(a+d, c, b)$ is bar-special, thus $\alpha \sim_{I} \beta$ by lemma \ref{LemmaTripleNotId}).
	\item If $(a, b, c, d) = (-2c, 4c, c, -3c)$, then $(a+d, c, b) = (-5c, c, 4c)$ and $(a+b, d, c) = (2c, -3c, c)$. If $p > 5$, then both are bar-regular and $\bar{a} \neq \bar{d}, \bar{b}$, thus (iii) applies; otherwise, $(a+d, c, b)$ is bar-special (since $-5\bar{c} = \bar{c} + 4\bar{c}$), thus $\alpha \sim_{I} \beta$ by lemma \ref{LemmaTripleNotId}.
	\item If $(a, b, c, d) = (a, -2a, -a, 2a)$ and $p > 5$, then the equivalence $\gamma \sim_{I} \alpha$ holds; if $p = 5$, then $\alpha$ is an identity, since $\bar{d} + \bar{a} = \bar{b}$.
\end{itemize}
The rest of step 2 holds without variations. About step 3, case I remains unchanged and, assuming $(d+a, b, c)$ bar-regular and $d \neq a$ in step II, we reach the same conclusion. About the proof that this assumption is not restrictive, we have to analyse the special cases.
\begin{itemize}
	\item (1)-ii: $(a, b, c, d) = (4c, -2c, c, -3c)$. If $p > 5$, then $d \neq b$ and $(d+b, a, c)$ is bar-regular. We observe that, if $p = 7$, then $\bar{d} = \bar{a}$, thus $\alpha$ is an identity. If $p = 5$, then $(a, b, c, d) = (-c, -2c, c, 2c)$. In this case, the only possibilities for $\eta$ that are not identities are $[x_{4}^{2c}, x_{2}^{-2c}, x_{1}^{-c}, x_{3}^{c}]$, that is equivalent to $\alpha$ by lemma \ref{LemmaQuadrNotId}-i; $[x_{4}^{2c}, x_{2}^{-2c}, x_{3}^{c}, x_{1}^{-c}]$, that is equivalent to the previous one by lemma \ref{LemmaTripleNotId} (since $(0, c, -c)$ is bar-special); and $[x_{4}^{2c}, x_{3}^{c}, x_{1}^{-c}, x_{2}^{-2c}]$, that is equivalent to the previous one because of identity \eqref{PolId3Car5}.
	\item (1)-iii: $(a, b, c, d) = (-2c, -2c, c, 3c)$. If $p > 5$, then $a \neq c$ and $(a+c, b, d)$ is bar-regular. If $p = 5$, then $\bar{d} = \bar{a}$, thus $\alpha$ is an identity.
	\item (1)-vi: $(a, b, c, d) = (-2d, 2d, -d, d)$. If $p > 5$, then $d \neq b$ and $(d+b, a, c)$ is bar-regular. If $p = 5$, then the only possibilities for $\eta$ that are not identities are $[x_{4}^{d}, x_{3}^{-d}, x_{1}^{-2d}, x_{2}^{2d}]$, that is equivalent to $\alpha$ because of identity \eqref{PolId3Car5}; $[x_{4}^{d}, x_{3}^{-d}, x_{2}^{2d}, x_{1}^{-2d}]$, that is equivalent to the previous one by lemma \ref{LemmaTripleNotId} (since $(0, 2d, -2d)$ is bar-special); and $[x_{4}^{d}, x_{2}^{2d}, x_{3}^{-d}, x_{1}^{-2d}]$, that is equivalent to the previous one by lemma \ref{LemmaQuadrNotId}-i.
	\item (2)-i: $(a, b, c, d) = (a, -2a, -2a, 3a)$. If $p > 5$, then $a \neq b$ and $(a+b, c, d)$ is bar-regular. If $p = 5$, then $(\bar{a}, \bar{b}, \bar{c}, \bar{d}) = (\bar{a}, -2\bar{a}, -2\bar{a}, -2\bar{a})$, so that the only possibility for $\eta$ that is not an identity is $[x_{4}^{3a}, x_{1}^{a}, x_{3}^{-2a}, x_{2}^{-2a}]$, that is equivalent to $\alpha$ by lemma \ref{LemmaTripleNotId}.
	\item (4)-i: $(a, b, c, d) = (-d, 2d, -2d, d)$. If $p > 5$, then $d \neq b$ and $(d+b, a, c) = (3d, -d, -2d)$ is bar-regular. If $p = 5$, then the only possibilities for $\eta$ that are not identities are $[x_{4}^{d}, x_{1}^{-d}, x_{3}^{-2d}, x_{2}^{2d}]$, that is equivalent to $\alpha$ by lemma \ref{LemmaTripleNotId}; $[x_{4}^{d}, x_{2}^{2d}, x_{1}^{-d}, x_{3}^{-2d}]$, that is equivalent to the previous one by lemma \ref{LemmaQuadrNotId}-i; and $[x_{4}^{d}, x_{3}^{-2d}, x_{2}^{2d}, x_{1}^{-d}]$, that is equivalent to the previous one because of identity \eqref{PolId3Car5}.
	\item (4)-ii: $(a, b, c, d) = (2c, -c, c, -2c)$. If $p > 5$, then $a \neq c$ and $(a+c, b, d) = (3c, -c, -2c)$ is bar-regular. If $p = 5$, then the only possibilities for $\eta$ that are not identities are $[x_{4}^{-2c}, x_{1}^{2c}, x_{3}^{c}, x_{2}^{-c}]$, that is equivalent to $\alpha$ by lemma \ref{LemmaTripleNotId}; $[x_{4}^{-2c}, x_{3}^{c}, x_{1}^{2c}, x_{2}^{-c}]$, that is equivalent to the previous one by lemma \ref{LemmaQuadrNotId}-i; and $[x_{4}^{-2c}, x_{2}^{-c}, x_{3}^{c}, x_{1}^{2c}]$, that is equivalent to the previous one because of identity \eqref{PolId3Car5}.
\end{itemize}
The extension to identities of any degree can be realized similarly to the case of zero characteristic, as stated in \cite[Theorem 4.12]{FDK}.


\section*{Aknowledgements}

I am grateful to my friend and colleague Dimas Gonçalves for having introduced me to p.i.-algebras and for many useful discussions.


\end{document}